\documentclass{article}
\usepackage{graphicx} 

\usepackage{setspace}

\usepackage{amsmath}
\usepackage{amssymb}
\usepackage{amsthm}
\usepackage{mathtools}
\usepackage{empheq}
\usepackage{enumitem}
\usepackage{titlesec}
\usepackage{graphicx}
\usepackage{caption}
\usepackage{subcaption}
\usepackage{diagbox}
\usepackage{hyperref}

\usepackage[T1]{fontenc}
\usepackage[utf8]{inputenc}
\usepackage{comment}
\usepackage{indentfirst}
\usepackage[top=1.25in,left=1.25in,right=1.25in]{geometry}

\numberwithin{equation}{section}

\title{\Large{\uppercase{\bf Classification of polyhedral graphs by numbers of common neighbours}}}
\author{\Large{Riccardo W. Maffucci}}
\date{}
\newcommand{\Addresses}{  
		R.W.~Maffucci, \textsc{Dipartimento di Matematica, Universit\`a di Torino\\\indent Via Carlo Alberto 10, Turin 10123, Italy}\par\nopagebreak\vspace{-0.35cm}
		\textit{E-mail address}, R.W.~Maffucci: \href{mailto:riccardowm@hotmail.com}{\texttt{riccardowm@hotmail.com}}
  }

\def\cc{\mathcal{C}}
\def\cg{\mathcal{G}}
\def\cw{\mathcal{W}}
\def\ks{\mathfrak{S}}

\def\ecc{\text{ecc}}
\def\rad{\text{rad}}
\def\diam{\text{diam}}

\newtheorem{thm}{Theorem}[section]
\newtheorem{lemma}[thm]{Lemma}
\newtheorem{prop}[thm]{Proposition}
\newtheorem{cor}[thm]{Corollary}
\newtheorem{defin}[thm]{Definition}

\begin{document}
\titleformat{\section}
  {\Large\scshape}{\thesection}{1em}{}
\titleformat{\subsection}
  {\large\scshape}{\thesubsection}{1em}{}
\maketitle
\Addresses

\begin{abstract}
We propose a classification of polyhedra (planar, $3$-connected graphs) according to their type i.e., their set of quantities of common neighbours for each pair of distinct vertices.

For every (finite) set of non-negative integers, we either classify all the polyhedra of that type, or construct infinitely many polyhedra of that type, or prove that none exist.

This problem is related to the theory of strongly regular and Deza graphs, distances in graphs, and degree sequences. There is potential for application to complex networks and data science.
\end{abstract}
{\bf Keywords:} Planar graph, Polyhedron, Vertex neighbours, Common neighbours, Dominating vertex, Strongly regular graph, Deza graph, Graph radius, Graph diameter, Graph transformation, Graph algorithm, $3$-polytope, Vertex degree, Degree sequence.
\\
{\bf MSC(2020):} 05C10, 05C75, 05C69, 05C12, 05E30, 52B05, 05C85.

\tableofcontents

\section{Introduction}
This paper deals with finite, undirected graphs $G=(V,E)$ with no multiple edges or loops. We will use the notation $N(u)$ for the set of neighbours (i.e, adjacent vertices) of $u\in V(G)$, and $N(u,v)$ for the set of common neighbours of the vertices $u$ and $v$.

A polyhedron (also called $3$-polytope) is a planar, $3$-connected graph. Polyhedral graphs up to isomorphism correspond naturally to polyhedral solids up to homeomorphism. A polyhedral graph may be embedded in the surface of a sphere in a unique way \cite{whit32}, hence via stereographic projection also in the plane, once an external region has been chosen. For this reason, we will always consider polyhedra together with their embedding in the plane. For further details and properties motivating the study of polyhedral graphs, we refer the interested reader to \cite[Introduction]{maffucci2025regularity}.

A regular graph is strongly regular if there exist non-negative integers $\lambda,\mu$ such that every pair of adjacent vertices has $\lambda$ common neighbours, and every pair of non-adjacent vertices has $\mu$ common neighbours. A regular graph is a Deza graph if there exist non-negative integers $\lambda,\mu$ such that every pair of vertices has either $\lambda$ or $\mu$ common neighbours. Thus Deza graphs are a generalisation of strongly regular graphs. Equivalently, a regular graph $G$ is a Deza graph if $|A|\leq 2$, where one defines
\begin{equation}
	\label{eq:A}
A=A(G):=\{a : \exists u,v\in V(G) \text{ satisfying } |N(u,v)|=a\}.
\end{equation}
In \cite[Theorem 1.1 and Corollary 1.3]{maffucci2025classification}, we have classified all the planar Deza graphs. For the theory of Deza graphs we refer the reader to the seminal paper \cite{erickson1999deza} and to the subsequent \cite{gavrilyuk2014vertex,goryainov2021deza}. One of the most famous and intriguing open problems in this area may be `the missing Moore graph(s)' \cite{dalfo2019survey}.

In applications of graph theory such as complex networks and data science, the number of common neighbours has been investigated as a measure of similarity for vertices \cite{wang2008study,papadopoulos2015network,wang2015opinion,yao2016link,li2018similarity,wang2022common}. For the related concept of neighbourhood complexity see \cite{reidl2019characterising}, and specifically for planar graphs \cite{joret2024neighborhood}.

We say that a graph $G$ is \textbf{of type} $\bf{A}$ if $A=A(G)$. That is to say, here we will use the abbreviation `type' for the concept `set of quantities of common neighbours for each pair of distinct vertices'. In \cite[Theorem 1.2]{maffucci2025classification}, we have classified all regular polyhedra according to their types. Outside of a few special cases, for a regular polyhedron $G$ the set $A$ is either $\{0,1\}$, or $\{0,1,2\}$, or $\{0,1,2,3\}$ according to whether $G$ contains subgraphs isomorphic to the square and/or $K(2,3)$.

The set $A$ may be defined for any graph, whether regular or not. In this paper, we pose the richer question of {\bf classifying all polyhedra according to their types}, thereby generalising \cite[Theorem 1.2]{maffucci2025classification}. For every finite set of non-negative integers $A$, we either characterise all infinitely many or finitely many solutions $G$ verifying $A=A(G)$, or we construct a wide class of polyhedra $G$ with $A=A(G)$, or we prove that none exist. To describe our results in detail, we need to first define a few classes of polyhedra.

Let
\[T_\ell=P_\ell+K_2,\]
depicted in Figure \ref{fig:tl}, where $P_\ell$ is the path on $\ell$ vertices, and $+$ is the graph join operator. Note that $T_1$ is the triangle, $T_2$ the tetrahedron, and $T_3$ the triangular bipyramid. For $\ell\geq 2$, $T_\ell$ is a polyhedron.
\begin{figure}[ht]
	\centering
	\includegraphics[width=3.cm]{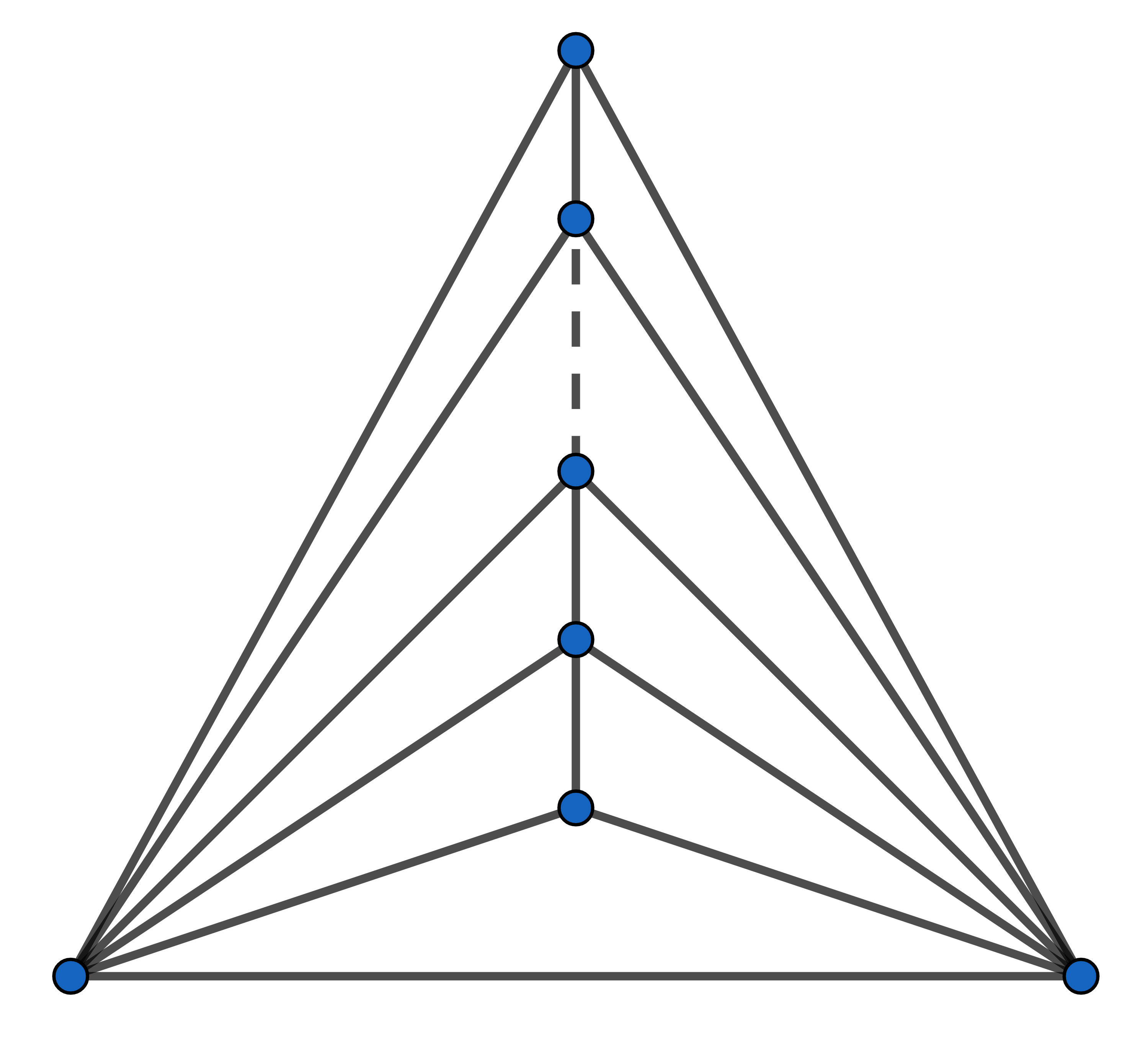}
	\caption{$T_\ell$.}
	\label{fig:tl}
\end{figure}

We define the family of polyhedra
\[\ks_1:=\{\text{bipyramids}\}\cup\{T_\ell : \ell\geq 2\}\cup\{S_i : 1\leq i\leq 10\},\]
where $S_1,S_2,\dots,S_{10}$ are the graphs in Figure \ref{fig:s}. Note that the octahedron is the square bipyramid.
\begin{figure}[ht]
	\centering
	\begin{subfigure}{0.19\textwidth}
		\centering
		\includegraphics[width=2.25cm]{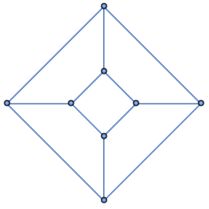}
		\caption{$S_1$.}
		\label{fig:s1}
	\end{subfigure}
	\hfill
	\begin{subfigure}{0.19\textwidth}
		\centering
		\includegraphics[width=2.25cm]{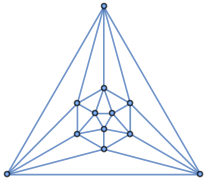}
		\caption{$S_2$.}
		\label{fig:s2}
	\end{subfigure}
	\hfill
	\begin{subfigure}{0.19\textwidth}
		\centering
		\includegraphics[width=2.25cm]{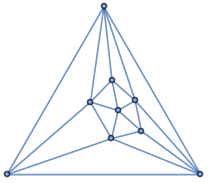}
		\caption{$S_3$.}
		\label{fig:s3}
	\end{subfigure}
	\hfill
	\begin{subfigure}{0.19\textwidth}
		\centering
		\includegraphics[width=2.25cm]{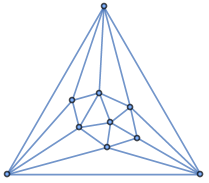}
		\caption{$S_4$.}
		\label{fig:s4}	
	\end{subfigure}
	\hfill
	\begin{subfigure}{0.19\textwidth}
		\centering
		\includegraphics[width=2.25cm]{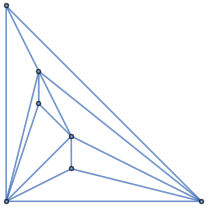}
		\caption{$S_5$.}
		\label{fig:s5}	
	\end{subfigure}
	\hfill
	\begin{subfigure}{0.19\textwidth}
		\centering
		\includegraphics[width=2.25cm]{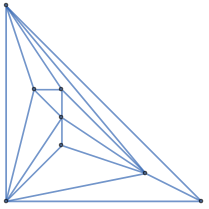}
		\caption{$S_6$.}
		\label{fig:s6}
	\end{subfigure}
	\hfill
	\begin{subfigure}{0.19\textwidth}
		\centering
		\includegraphics[width=2.25cm]{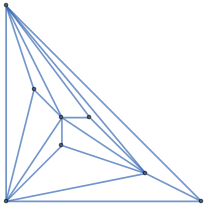}
		\caption{$S_7$.}
		\label{fig:s7}
	\end{subfigure}
	\hfill
	\begin{subfigure}{0.19\textwidth}
		\centering
		\includegraphics[width=2.25cm]{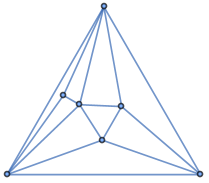}
		\caption{$S_8$.}
		\label{fig:s8}
	\end{subfigure}
	\hfill
	\begin{subfigure}{0.19\textwidth}
		\centering
		\includegraphics[width=2.25cm]{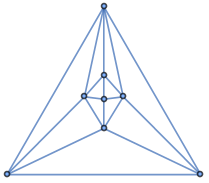}
		\caption{$S_9$.}
		\label{fig:s9}	
	\end{subfigure}
	\hfill
	\begin{subfigure}{0.19\textwidth}
		\centering
		\includegraphics[width=2.25cm]{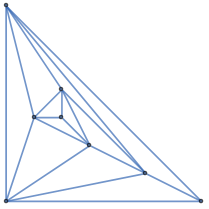}
		\caption{$S_{10}$.}
		\label{fig:s10}	
	\end{subfigure}
	\caption{The ten exceptional graphs.}
	\label{fig:s}
\end{figure}

\begin{defin}
	We call $\cw_3$ the class of polyhedra obtained from an $n$-gonal pyramid (i.e., wheel graph), $n\geq 5$, by adding a certain number of pairwise independent edges (at least one).
	\\
	We call $\cw_4$ the class of polyhedra $G$ obtained from a pyramid of base
	\[[v_1,v_2,\dots,v_n], \qquad n\geq 5\]
	by adding pairwise disjoint cycles (at least one) such that all of the following hold. Each cycle is of length different from $4$; if one of the cycles has length $3$, and $v_i,v_j$, $i<j$ belong to this cycle, then $4\leq j-i\leq n-4$; if $v_iv_j\in E(G)$, $i<j$, then $v_{i+1}v_{j-1}\not\in E(G)$.
\end{defin}
A member of $\cw_3$ appears in Figure \ref{fig:w3}.
\begin{figure}[ht]
	\centering
	\includegraphics[width=4.cm]{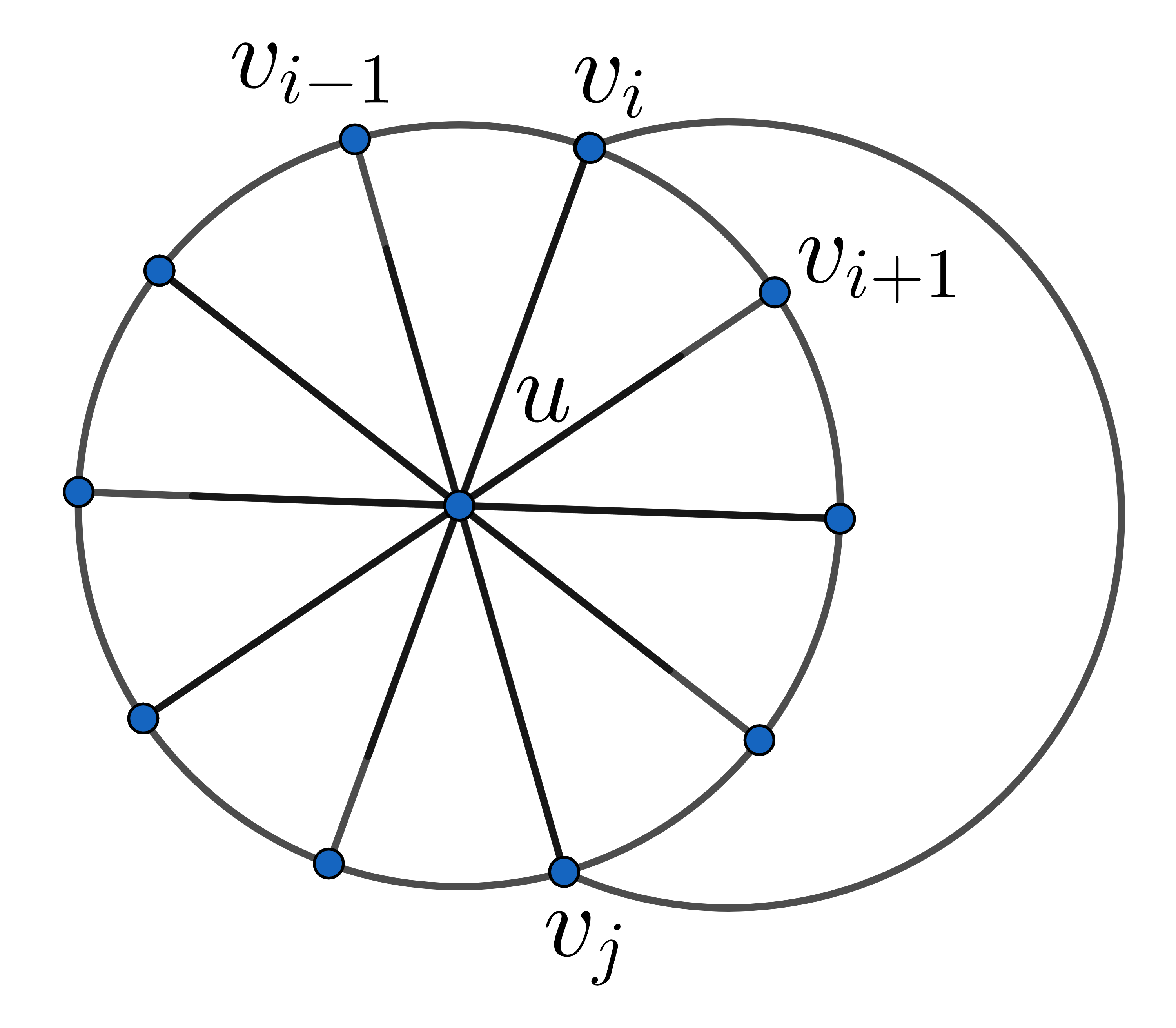}
	\caption{A member of $\cw_3$.}
	\label{fig:w3}
\end{figure}

\begin{defin}
	\label{def:B}
	Let $n\geq 4$ be even. We define the graph $B_n$ where we start with the cycle
	\[b_1,b_2,\dots,b_{2n}\]
	and add the edges $xb_i$ for $1\leq i\leq 2n$ and
	\[yb_j \qquad j\equiv 1,2\pmod 4, \quad 1\leq j\leq 2n.\]
	We also define the graph $B'_n$ where we start with the path
	\[b_1,b_2,\dots,b_{2n-2}\]
	and add the edges $xb_i$ for $1\leq i\leq 2n-2$ and
	\[yb_j \qquad j\equiv 1,2\pmod 4, \quad 1\leq j\leq 2n-2.\]
\end{defin}
Illustrations may be found in Figure \ref{fig:b}. These families of graphs verify $A(B_n)=A(B'_n)=\{1,2,n\}$. Note that the defined graphs are polyhedra, and $B_n-y$ is isomorphic to the $2n$-gonal pyramid.
\begin{figure}[ht]
	\centering
	\begin{subfigure}{0.48\textwidth}
		\centering
		\includegraphics[width=6.5cm]{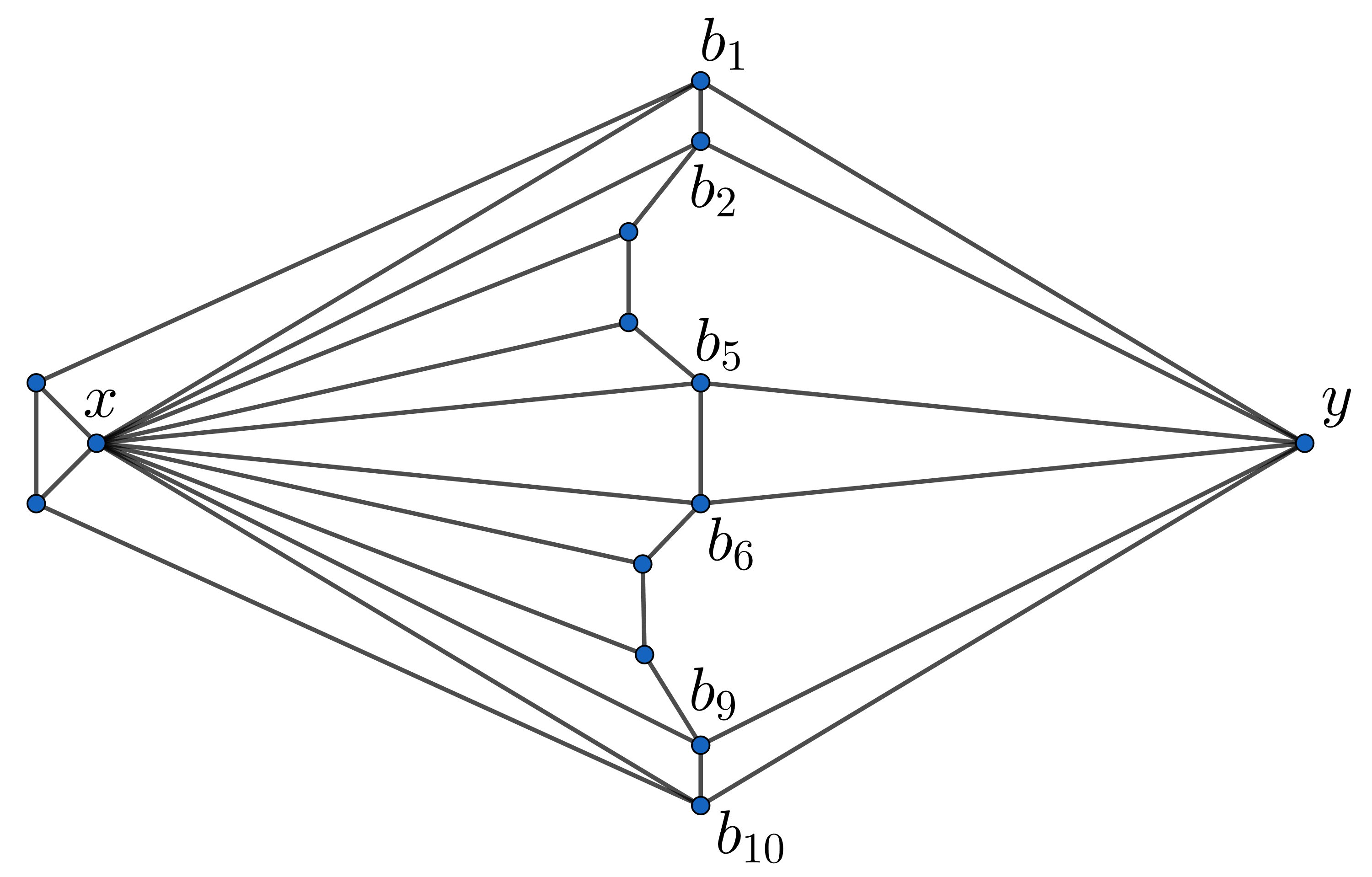}
		\caption{$B_{6}$.}
		\label{fig:bn}
	\end{subfigure}
	\hfill
	\begin{subfigure}{0.48\textwidth}
		\centering
		\includegraphics[width=6.5cm]{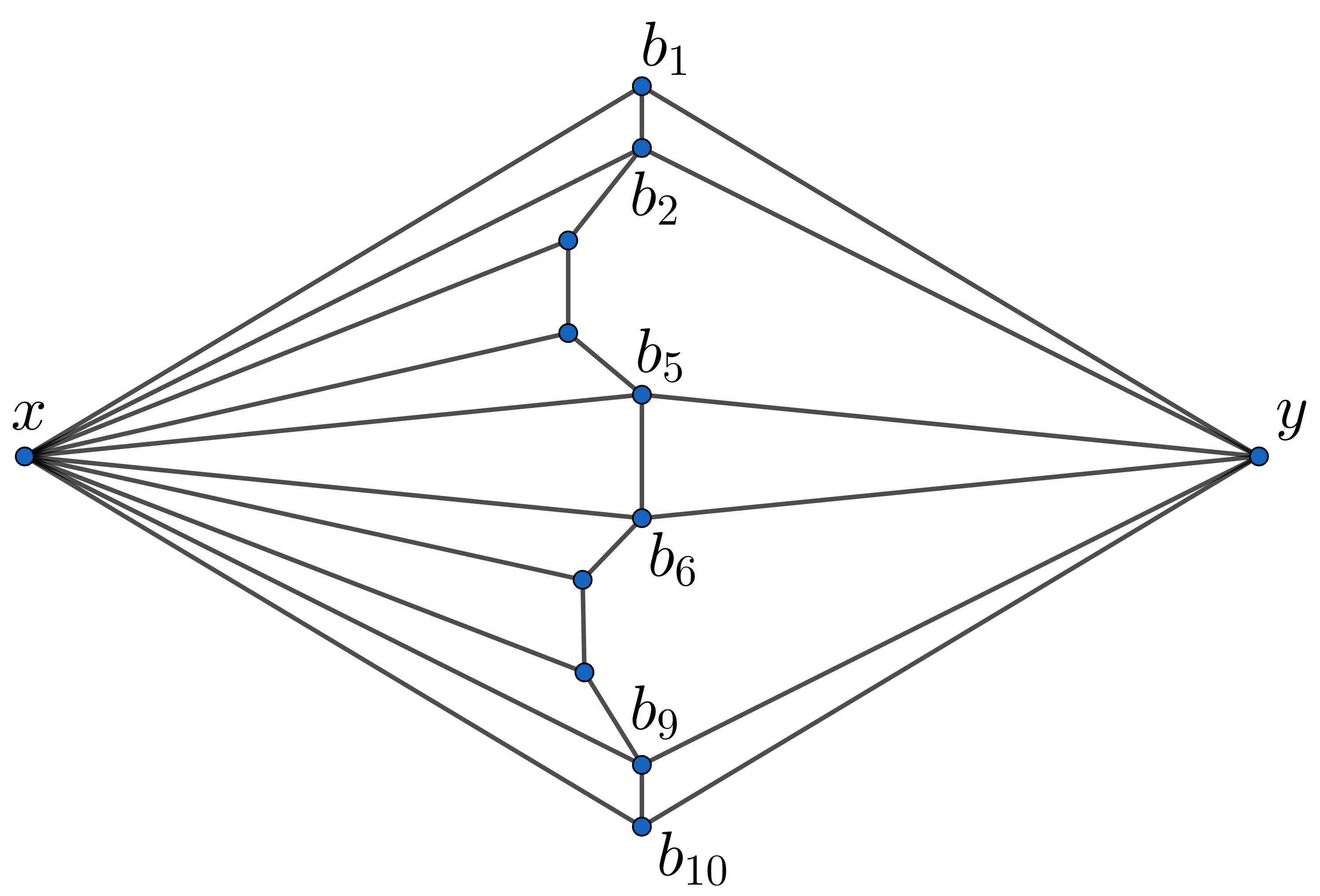}
		\caption{$B'_{6}$.}
		\label{fig:bnp}
	\end{subfigure}
	\caption{Illustrations of $B_6$ and $B'_6$.}
	\label{fig:b}
\end{figure}

We are ready to state our main results.
\begin{thm}
	\label{thm:1}
A polyhedron $G$ satisfies $1\not\in A$ if and only if $G\in\ks_1$. We have the classification of Table \ref{t:1}.
\begin{table}[ht]
	\centering
	$\begin{array}{|c|c|}
		\hline A(G)&G\\
		\hline \{2\}&\text{tetrahedron }T_2\\
		\hline \{0,2\}&\text{cube } S_1, \text{ icosahedron } S_2\\
		\hline \{2,3\}&T_3, S_3\\
		\hline \{2,4\}&\text{octahedron}\\
		\hline \{0,2,3\}&S_4\\
		\hline \{2,3,4\}&T_4,S_5,S_6,S_7,S_8,S_{9}\\
		\hline \{2,3,\ell\}, \ \ell\geq 5&\ell\text{-gonal bipyramid}, \ T_{\ell}\\
		\hline \{0,2,3,4\}&S_{10}\\
		\hline
	\end{array}$
	\caption{All polyhedra $G$ satisfying $1\not\in A$.}
	\label{t:1}
\end{table}

\end{thm}

Theorem \ref{thm:1} will be proven in Section \ref{sec:no1}.

\begin{thm}
\label{thm:2}
Let $G$ be a polyhedron of order $p$ satisfying $1\in A$, and $A'_n$ be a (finite) non-empty set of integers $\geq n$. Then $G$ is of one of the types listed in the first column of Table \ref{t:2}, and we have the corresponding conditions on $G$ stated in the second column.
\begin{table}[ht]
	\centering
	$\begin{array}{|c|c|}
		\hline A(G)&G\\
		\hline \{1,2\}&\Rightarrow p\leq 24 \text{ or } n\text{-gonal pyramid, } n\geq 5\\
		\hline \{1,2,3\}&\Rightarrow p\leq 47 \text{ or}\in\cw_3\\
		\hline \{1,2,4\}&\Rightarrow p\leq 78 \text{ or}\in\cw_4\\
		\hline \{1,2,\ell\}, \text{ even }\ell\geq 6&\Leftrightarrow B_\ell, B'_\ell\\
		\hline \{1,2,3\}\cup A'_4&\text{infinitely many polyhedra for each }A'_4\\
		\hline \{1,2,4\}\cup A'_5&\text{infinitely many polyhedra for each }A'_5\\
		\hline\{0,1\}&
		\Leftrightarrow\text{no $4$-cycles}\\
		\hline \{0,1,2\}&\text{infinitely many polyhedra; refer to Corollary \ref{cor:012}}\\
		\hline \{0,1,2\}\cup A'_3&\text{infinitely many polyhedra for each }A'_3\\
		\hline
	\end{array}$
	\caption{All types of polyhedra $G$ satisfying $1\in A$.}
	\label{t:2}
\end{table}

\end{thm}
Together, Theorems \ref{thm:1} and Theorems \ref{thm:2} completely classify all types of polyhedra. Theorem \ref{thm:2} will be proven in Sections \ref{sec:pre}, \ref{sec:rad1}, \ref{sec:compl}, and \ref{sec:wide}, by presenting intermediate results that will populate Table \ref{t:2}.

For the types $\{1,2\}$, $\{1,2,3\}$, $\{1,2,4\}$, and $\{1,2,\ell\}$ with even $\ell\geq 6$, in Sections \ref{sec:rad1} and \ref{sec:compl} we will prove that, outside a finite set of solutions, all remaining ones are either pyramids or belong to one of the well-understood classes $\cw_3$, $\cw_4$, $B_\ell$, $B'_\ell$.

For each of the remaining types in Table \ref{t:2} $\{0,1\}$, $\{0,1,2\}$, $\{1,2,3\}\cup A'_4$, $\{1,2,4\}\cup A'_5$, and $\{0,1,2\}\cup A'_3$ there exist wide classes of polyhedra of each type. For the cases $\{0,1\}$ and $\{0,1,2\}$, this was shown in \cite{maffucci2025classification}. For instance one has the following explicit classes. All polyhedra of girth $5$ (equivalently, all polyhedra of minimum face length $5$) verify $A=\{0,1\}$. With the exception of the cube, all $3$-regular (i.e., cubic) polyhedra containing a quadrangular face verify $A=\{0,1,2\}$ \cite[Theorem 1.2]{maffucci2025classification}.

Here we will greatly strengthen the results for the types $\{0,1\}$ and $\{0,1,2\}$. In Section \ref{sec:pre}, we will show that a polyhedron satisfies $A=\{0,1\}$ if and only if it does not contain $4$-cycles. In Section \ref{sec:rad1}, we will give necessary and sufficient conditions for a polyhedron on at least $25$ vertices to verify $A=\{0,1,2\}$.

In Section \ref{sec:wide}, for each of the types $\{1,2,3\}\cup A'_4$, $\{1,2,4\}\cup A'_5$, and $\{0,1,2\}\cup A'_3$ we will construct wide classes of polyhedra of the corresponding type. For all sets of non-negative integers $A$ that do not appear in Tables \ref{t:1} and \ref{t:2}, we will prove in Section \ref{sec:compl} that there are no polyhedra $G$ satisfying $A=A(G)$.


\paragraph{Preliminaries.}
A graph is $k$-connected if it has at least $k+1$ vertices, and however one removes fewer than $k$ of them, the resulting graph is connected.
\\
A graph is planar if it may be embedded in the plane with edges crossing only at vertices. A plane graph is a planar graph considered together with an embedding in the plane. Polyhedral graphs have a unique embedding up to choosing the external region, hence we tacitly consider them together with their embedding.
\\
An outerplanar graph is a planar graph such that there exists a planar embedding where all vertices lie on the boundary of one region.
\\
For a (non-trivial) connected graph $G$, the eccentricity of a vertex $u$
\[\ecc(u)\]
is the maximal distance from $u$ to $v\in V(G)$. The greatest eccentricity among all vertices is the diameter of $G$
\[\diam(G).\]
Equivalently, $\diam(G)$ is the maximal distance between any pair of vertices in $G$. The radius of $G$
\[\rad(G)\]
is the minimal vertex eccentricity. It is well-known that for any connected graph $G$ we have
\[\rad(G)\leq\diam(G)\leq 2\rad(G).\]
A graph has diameter $1$ if and only if it is a complete graph. A graph has radius $1$ if and only if there exists a vertex adjacent to all other vertices, also called a dominating vertex.
\\
Given a plane graph $G$ and $u\in V(G)$, we define the plane neighbourhood of $u$ in $G$ to be the subgraph $\Gamma_u(G)$ of $G$ defined by
\begin{equation}
\label{eq:plne}
E(\Gamma_u(G))=\{e\in E(G) : e \text{ lies on a region that contains } u\}.
\end{equation}
Everywhere the notation $A'_n$ indicates a finite, non-empty set of integers $\geq n$.

\paragraph{Plan of the paper.} Section \ref{sec:no1} is entirely dedicated to proving Theorem \ref{thm:1}. The next four sections are dedicated to proving Theorem \ref{thm:2}. In Section \ref{sec:pre} we will lay out the preliminary work. In Section \ref{sec:rad1} we will focus on the important class of polyhedra of graph radius $1$, that will be useful in what follows. In Section \ref{sec:compl} we will complete the classification of types of polyhedra stated in Table \ref{t:2}. In Section \ref{sec:wide} we will construct a wide class of polyhedra for a few of the types in \ref{t:2}. Code for these constructions appears in Appendix \ref{app:a}. The results used to prove Theorem \ref{thm:2} will be recapped at the end of Section \ref{sec:wide}.

\section{Types of polyhedra with $1\not\in A$: Proof of Theorem \ref{thm:1}}
\label{sec:no1}

Let $1\not\in A$. Our first goal is to show that $G$ is either the cube or a triangulation of the sphere (maximal planar graph). Assume by contradiction that $G$ contains a face
\[[u_1,u_2,\dots,u_n], \quad n\geq 5.\]
Let $u_{n+1}:=u_1$ and $u_{n+2}:=u_2$. Since $1\not\in A$, then for every $1\leq i\leq n$ the vertices $u_i$ and $u_{i+2}$ have a common neighbour other than $u_{i+1}$. By planarity, the only possibility is that there exists $v\in V(G)$ adjacent to all of $u_1,u_2,\dots,u_n$. Moreover, for every $1\leq i\leq n$, the vertices $u_i,u_{i+1}$ have a common neighbour $w_i\neq v$.

Next, for every $1\leq i\leq n$, the vertices $w_i,u_{i+2}$ have a common neighbour other than $u_{i+1}$. By planarity, the only possibility is that $v$ is adjacent to all of $w_1,w_2,\dots,w_n$, as in Figure \ref{fig:no1a}.
\begin{figure}[ht]
	\centering
	\begin{subfigure}{0.48\textwidth}
		\centering
		\includegraphics[width=4.cm]{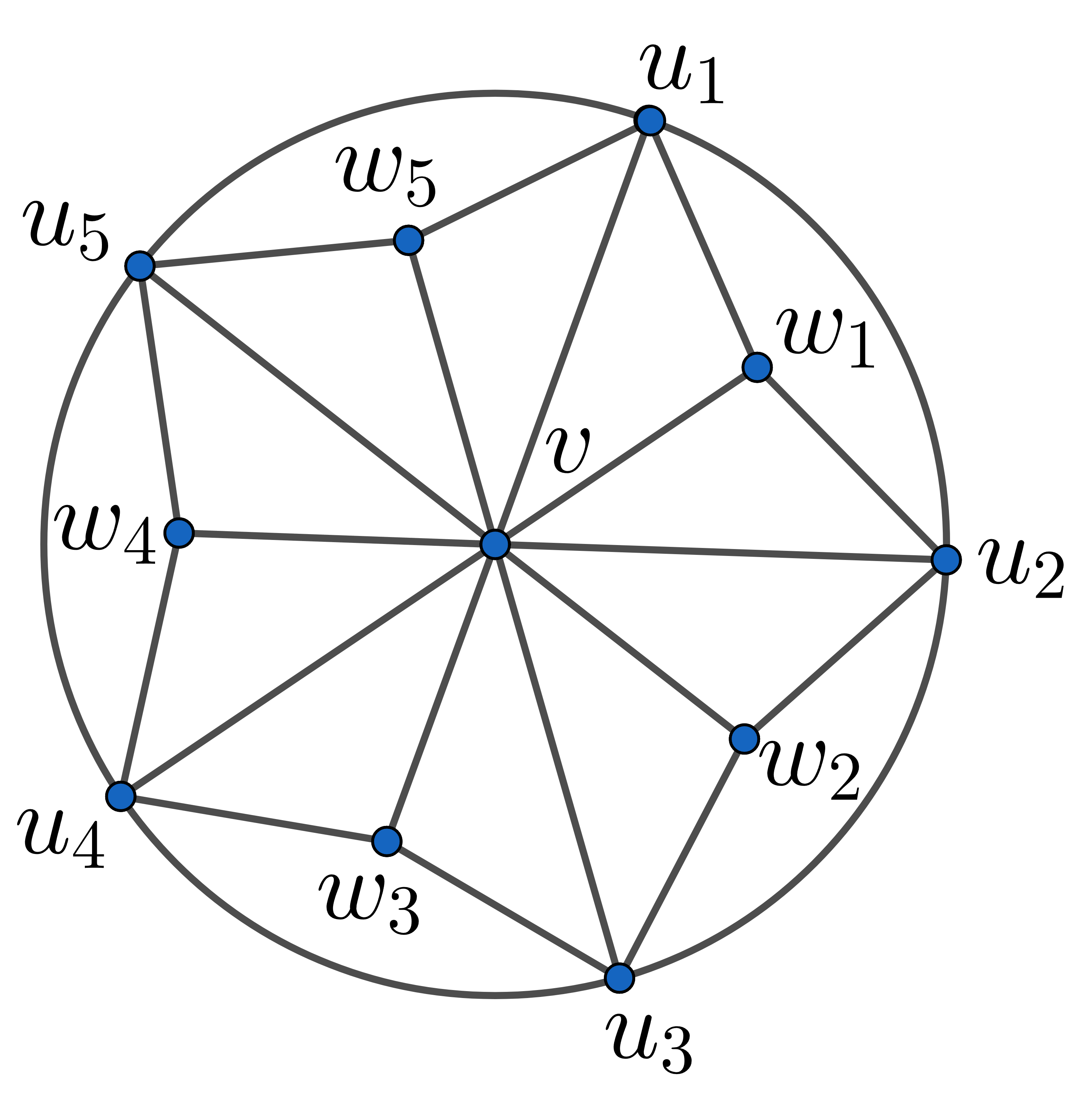}
		\caption{}
		\label{fig:no1a}
	\end{subfigure}
	\hspace{-1.cm}
	\begin{subfigure}{0.48\textwidth}
		\centering
		\includegraphics[width=3.5cm]{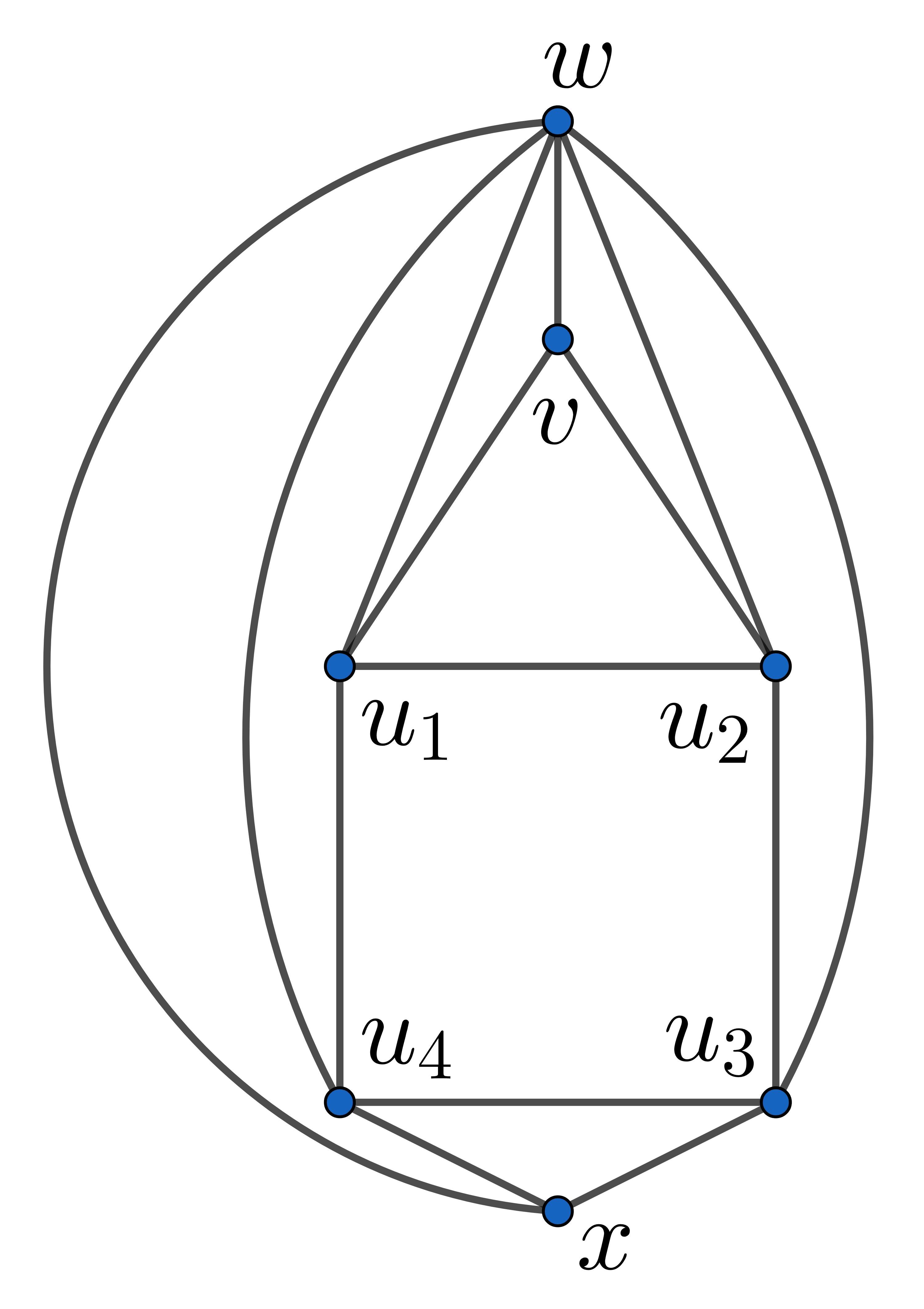}
		\caption{}
		\label{fig:no1b}
	\end{subfigure}
	\caption{Impossible situations for $1\not\in A$.}
	\label{fig:no1ab}
\end{figure}

Now by planarity,
\[w_1u_3,\ w_1u_4,\ w_3u_1,\ w_3u_2\not\in E(G),\]
thus $N(w_1,w_3)=\{v\}$, contradiction.

Next, assume by contradiction that $G$ contains adjacent faces
\[[u_1,u_2,u_3,u_4]\quad\text{ and }\quad[u_1,u_2,v].\]
Since $1\not\in A$, there exists $w\in V(G)$, $w\neq v$, adjacent to $u_1,u_2$.

Similarly, $v,u_4$ have a common neighbour other than $u_1$, and $v,u_3$ have a common neighbour other than $u_2$. Hence by planarity, $w$ is also adjacent to $v,u_3,u_4$. Now $u_3,u_4$ have a common neighbour $x\neq w$. Since $N(x,u_1)\neq\{u_4\}$, we deduce that $xw\in E(G)$, as in Figure \ref{fig:no1b}. By planarity, $N(v,x)=\{w\}$, contradiction.

It follows that $G$ is either a quadrangulation or a triangulation of the sphere. Let $G$ be a quadrangulation, and
\[[u_1,u_2,u_5,u_4]\quad\text{ and }\quad[u_2,u_3,u_6,u_5]\]
be adjacent faces. We write
\[N(u_1,u_3)\supseteq\{u_2,v\}\quad\text{ and }\quad N(u_4,u_6)\supseteq\{u_5,w\}.\]
We cannot have $v=w$, otherwise $u_1,u_4,v$ would be a triangle in $G$, but $G$ has no odd length cycles, since a quadrangulation of the sphere is a bipartite graph. Thus $v\neq w$. Now $v,u_4$ have a common neighbour other than $u_1$, and likewise $w,u_1$ have a common neighbour other than $u_4$. Since $G$ is planar and bipartite, we obtain $vw\in E(G)$. So far, we have drawn the graph of the cube. We claim that there are no more vertices in $G$. By contradiction, let $x$ be a new vertex. By the symmetries of the cube, we may assume w.l.o.g.\ that $x$ lies inside the $4$-cycle $v,u_1,u_2,u_3$, and is adjacent to $u_1$, as in Figure \ref{fig:no1c}. Since $G$ is planar and bipartite, $u_4$ cannot be adjacent to any of $u_2,u_3,v$, thus $N(x,u_4)=\{u_1\}$, contradiction. Hence $G$ is the cube.
\begin{figure}[ht]
	\centering
	\includegraphics[width=4.cm]{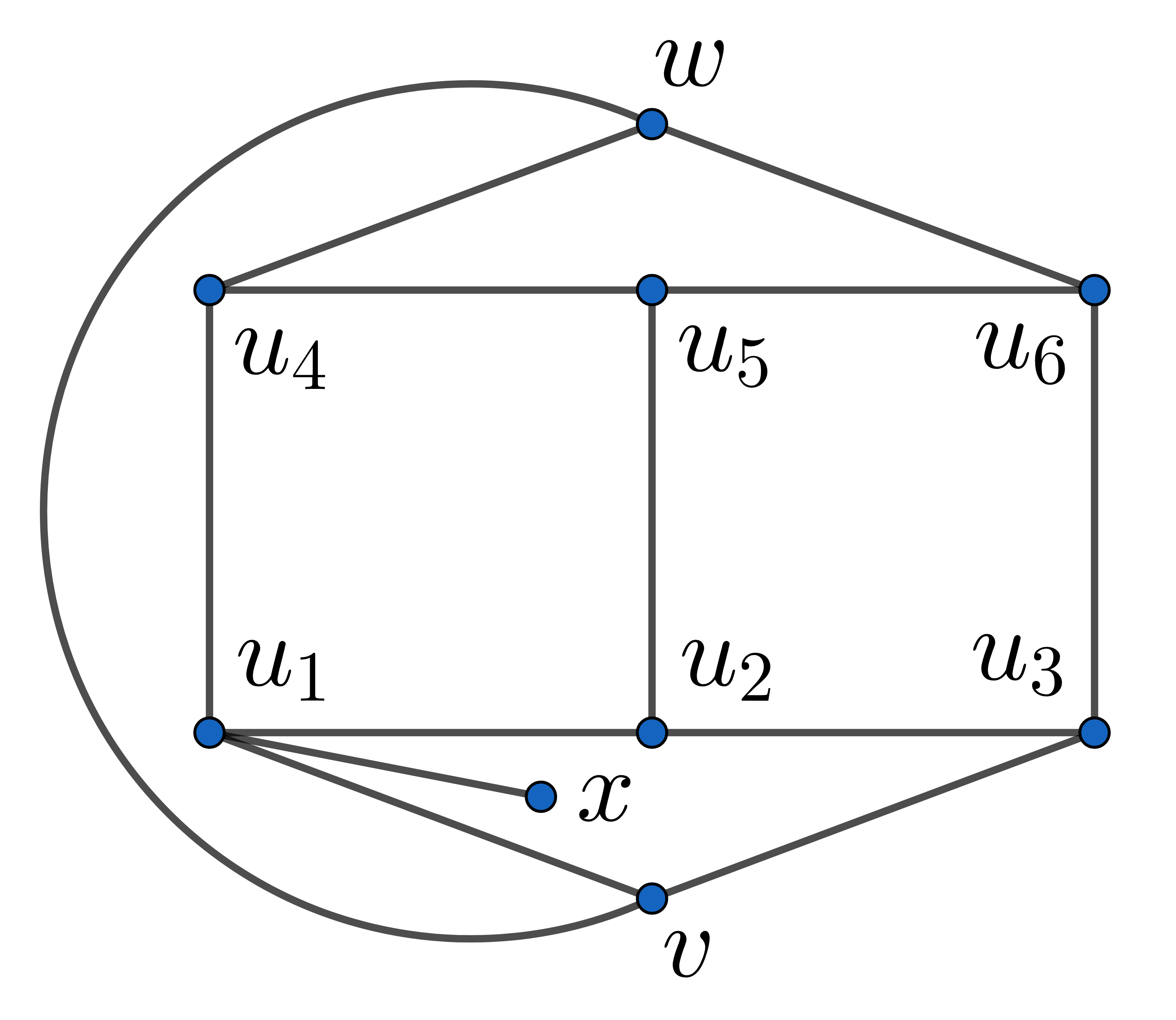}
	\caption{Case where $G$ is a quadrangulation and $1\not\in A$.}
	\label{fig:no1c}
\end{figure}

Henceforth we may assume that $G$ is a triangulation. Let $u\in V(G)$ be of maximal degree in $G$,
\[d=\deg(u)=\Delta(G),\]
and
\begin{equation}
	\label{eq:cyv}
v_1,v_2,\dots,v_{d}
\end{equation}
be its neighbours in cyclic order around $u$ in the planar immersion of $G$. Since $G$ is a triangulation, \eqref{eq:cyv} is a cycle of $G$ (in other words, the plane neighbourhood $\Gamma_u(G)$ is a $d$-gonal pyramid with apex $u$). We consider the indices modulo $d$, identifying $v_{-1}=v_{d-1}$, $v_0=v_d$, $v_{d+1}=v_1$, and so forth.

Suppose for the moment that $d\geq 7$. We claim that if there exist $1\leq i<i'\leq d$ with $2\leq i'-i\leq d-2$ such that $v_iv_{i'}\in E(G)$ as in Figure \ref{fig:no1d}, then there exists $j$ such that $v_j$ is adjacent to all of
\begin{equation}
	\label{eq:j}
v_1,v_2,\dots,v_{j-1},v_{j+1},v_{j+2},\dots,v_d.
\end{equation}
That is to say, the graph constructed so far is $T_{d-1}$ of Figure \ref{fig:tl}. To prove the claim, since $|N(v_{i+2},v_{i'+2})|\neq 1$, then by planarity both of $v_{i+2},v_{i'+2}$ are adjacent to either $v_i$ or $v_{i'}$, w.l.o.g.\ say $v_i$. Thus $v_iv_{i+2}\in E(G)$. We now consider
\[N(v_{i+1},v_\ell)\]
with
\begin{equation*}
\ell\in\{1,2,\dots,d\}\setminus\{i-1,i,i+1,i+2,i+3\}.
\end{equation*}
Since $|N(v_{i+1},v_\ell)|\neq 1$, then $v_\ell$ is adjacent to either $v_{i}$ or $v_{i+2}$. We cannot have $v_{i-2}v_i,v_{i+2}v_{i+4}\in E(G)$ (Figure \ref{fig:no1e}), otherwise by planarity $N(v_{i-1},v_{i+3})=\{u\}$ would follow. Hence either both of $v_{i-2},v_{i+4}$ are adjacent to $v_i$ or both are adjacent to $v_{i+2}$, say w.l.o.g.\ $v_i$. By planarity $v_i$ is adjacent to $v_1,v_2,\dots,v_{i-1},v_{i+1},v_{i+2},v_{i+4},v_{i+5},\dots,v_d$. Since $|N(v_{i-1},v_{i+3})|\neq 1$ and by planarity, $v_{i+3}v_i\in E(G)$ so that the claim is proven.
\begin{figure}[ht]
	\centering
	\begin{subfigure}{0.48\textwidth}
		\centering
		\includegraphics[width=4.25cm]{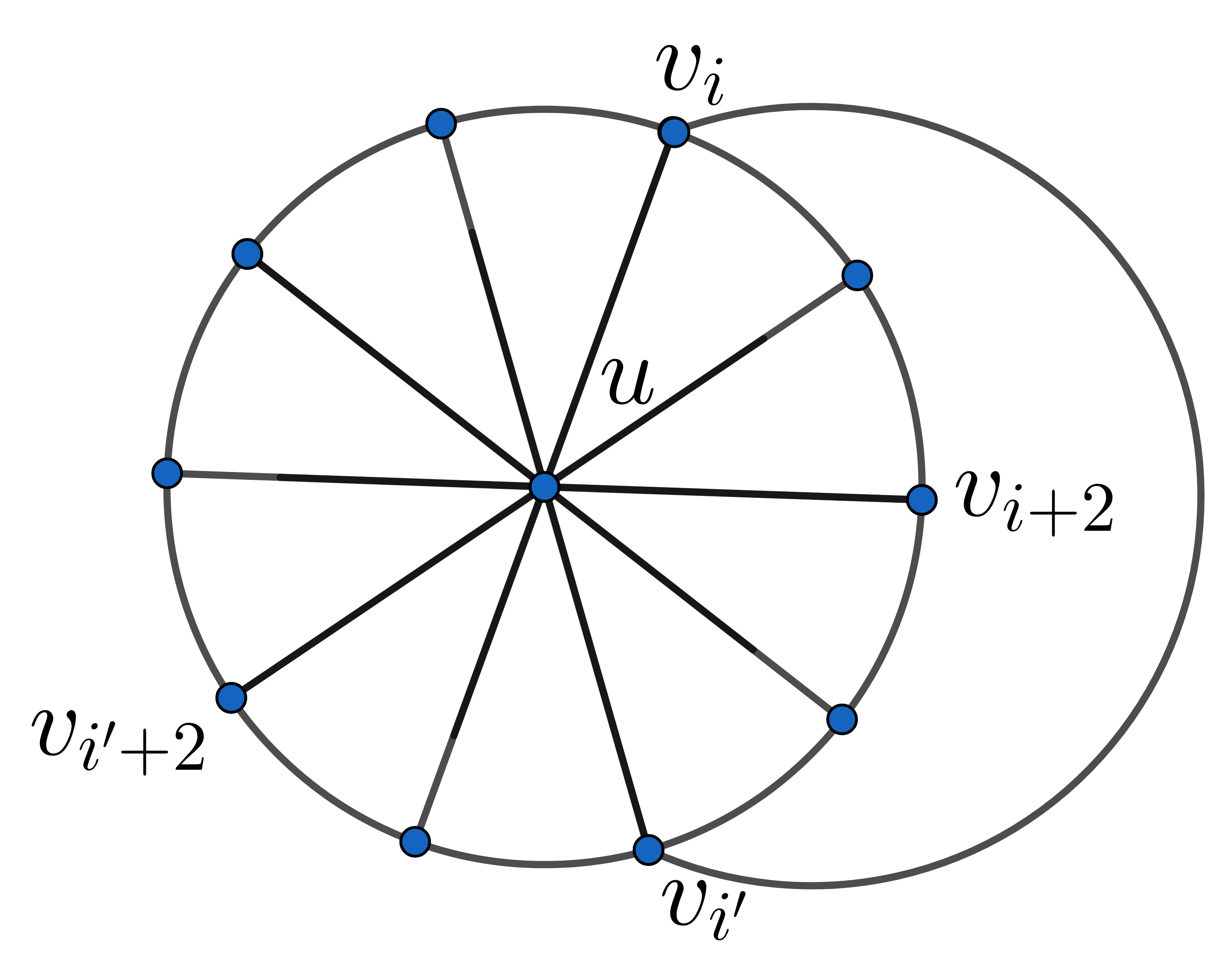}
		\caption{}
		\label{fig:no1d}
	\end{subfigure}
	\hspace{-1.cm}
	\begin{subfigure}{0.48\textwidth}
		\centering
		\includegraphics[width=4.cm]{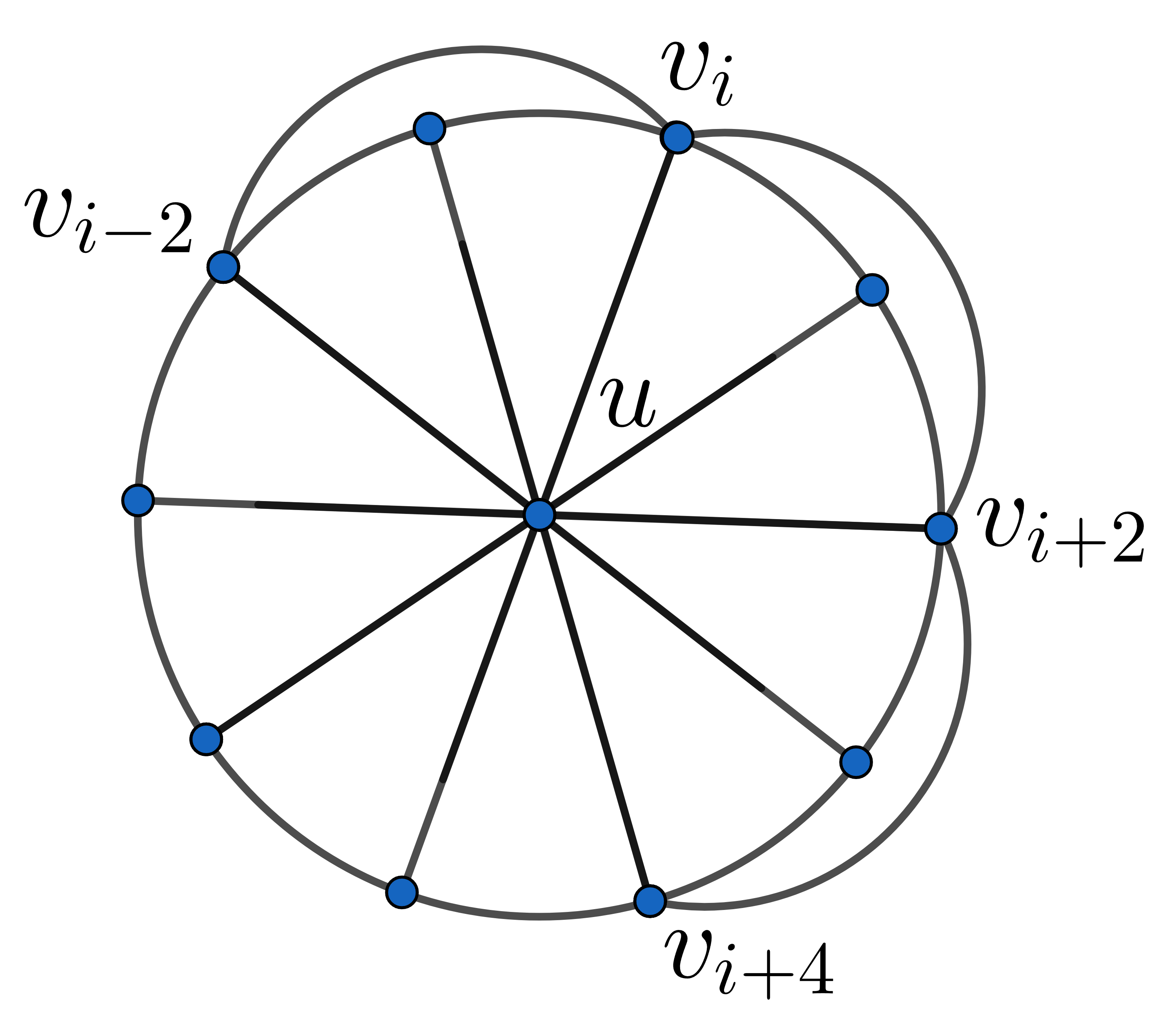}
		\caption{}
		\label{fig:no1e}
	\end{subfigure}
	\caption{There exists $j$ such that $v_j$ is adjacent to all of \eqref{eq:j}.}
	\label{fig:no1de}
\end{figure}

Now assume instead that for all $1\leq i<i'\leq d$ with $2\leq i'-i\leq d-2$ one has $v_iv_{i'}\not\in E(G)$. For every $1\leq i\leq d$, the vertices $v_i$ and $v_{i+3}$ have another common neighbour apart from $u$. Then for each $1\leq i\leq d$ 
there exists a vertex $w_i\neq u$ adjacent to $v_i$ and $v_{i+3}$. By planarity, we have in fact $w_1=w_2=\dots=w_d=:w$, and $w$ is adjacent to all of \eqref{eq:cyv}, so that the graph constructed so far is the $d$-gonal bipyramid. To summarise, if $G$ is a triangulation with $d=\Delta(G)\geq 7$, then the graph constructed so far is either the $d$-gonal bipyramid or $T_{d-1}$. It is easy to see, as in the case of quadrangulations, that for each of these possibilities there can be no more vertices, so that $G$ itself is either the $d$-gonal bipyramid or $T_{d-1}$.

Now let $d=6$. If for all $1\leq i<i'\leq d$ with $2\leq i'-i\leq d-2$ one has $v_iv_{i'}\not\in E(G)$, then since
\[|N(v_1,v_4)|,|N(v_2,v_5)|,|N(v_3,v_6)|\neq 1,\]
as above there exists $w$ adjacent to each of $v_1,v_2,v_3,v_4,v_5,v_6$, thus $G$ is the $6$-gonal bipyramid. Otherwise, as seen above we may assume w.l.o.g.\ that $v_1v_3\in E(G)$.
As $|N(v_2,v_5)|\neq 1$, we also have w.l.o.g.\ $v_5v_3\in E(G)$. As $G$ is a triangulation, we know that either $v_3v_6\in E(G)$, or $v_1v_5\in E(G)$, or there exists a new vertex inside of the cycle $\cc: v_1,v_3,v_5,v_6$. In the first case, so far we have the graph $T_{5}$. In the second case, so far we have $S_5$ of Figure \ref{fig:s5}.

Let us inspect the third case. By $3$-connectivity of $G$, at least three of $v_1,v_3,v_5,v_6$ are adjacent to vertices inside of $\cc$. W.l.o.g., we may take a new vertex $x$ with $xv_1\in E(G)$. Analysing $N(x,v_2)$, we see that $xv_3\in E(G)$. Now analysing $N(x,v_4)$, we see that $xv_5\in E(G)$. As before, since $G$ is a triangulation, either $xv_6\in E(G)$, or $v_1v_5\in E(G)$, or there exists a new vertex $y$ inside of the cycle $\cc': v_1,x,v_5,v_6$. The graph constructed so far is $S_6$ of Figure \ref{fig:s6} in the first sub-case and $S_7$ of Figure \ref{fig:s7} in the second sub-case. In the third sub-case, we have w.l.o.g.\ $yv_1\in E(G)$, so that by planarity $N(y,v_2)=\{v_1\}$, contradiction. Having constructed one of $T_5,S_5,S_6,S_7$, as above we see that no more vertices may be added while respecting $1\not\in A$. To recap, if $G$ is a triangulation with $d=\Delta(G)\geq 6$, then $G$ is either the $d$-gonal bipyramid, or $T_{d-1}$, or one of $S_5,S_6,S_7$.

Now instead let $G$ be a triangulation satisfying $\Delta(G)\leq 5$. Denote by $p_i$ the number of vertices in $G$ of degree $i$. Using the handshaking lemma and the fact that $G$ is maximal planar, we write
\[3p_3+4p_4+5p_5=2q=6p-12=6p_3+6p_4+6p_5-12,\]
whence
\[12=3p_3+2p_4+p_5\geq p_3+p_4+p_5=p.\]
We inspected all triangulations with up to $12$ vertices (code available on request), where we found, apart from bipyramids, $T_\ell$ for $2\leq\ell\leq 10$, and $S_5,S_6,S_7$, six more graphs satisfying $1\not\in A$, namely $S_i$ with $i=2,3,4,8,9,10$ in Figure \ref{fig:s}.

We have classified all polyhedra $G$ (hence also all types of polyhedra) satisfying $1\not\in A(G)$, according to Table \ref{t:1}. The proof of Theorem \ref{thm:1} is complete.

\section{Preliminaries for the scenario $1\in A$}
\label{sec:pre}
The ultimate goal of the rest of this paper is to prove Theorem \ref{thm:2}. In the present section, we will establish a few auxiliary lemmas concerning $A(G)$ for generic graphs $G$.

\begin{lemma}
	\label{le:yes2}
Let $G$ be a planar graph. If there exists $a\geq 3$ such that $a\in A$, then $2\in A$.
\end{lemma}
\begin{proof}
Let $a\geq 3$ such that $a\in A$, and suppose by contradiction that $2\not\in A$. We claim that $G$ contains a subgraph isomorphic to the square pyramid. By assumption, there exist distinct $u,v\in V(G)$ such that
\[N(u,v)\supseteq\{w_1,w_2,w_3\}.\]
Since $2\not\in A$, then $w_1,w_2$ have a common neighbour $x\not\in\{u,v\}$, as in Figure \ref{fig:no2a}. If $x=w_3$ then $G$ contains a square pyramid of apex $w_3$ and base $[u,w_1,v,w_2]$ and the claim is proven. Otherwise, we introduce the labels
\[u'=w_1, \quad v'=w_2, \quad w_1'=u, \quad w_2'=x,\quad w_3'=v.\]
Since $2\not\in A$, then $w_1',w_2'$ have a common neighbour $x'\not\in\{u',v'\}$. If $x'=w_3'$ the claim is proven, otherwise we continue in the same fashion, so that by finiteness, $G$ indeed contains a square pyramid as a subgraph.
\begin{figure}[ht]
	\centering
	\begin{subfigure}{0.48\textwidth}
		\centering
		\includegraphics[width=4.5cm]{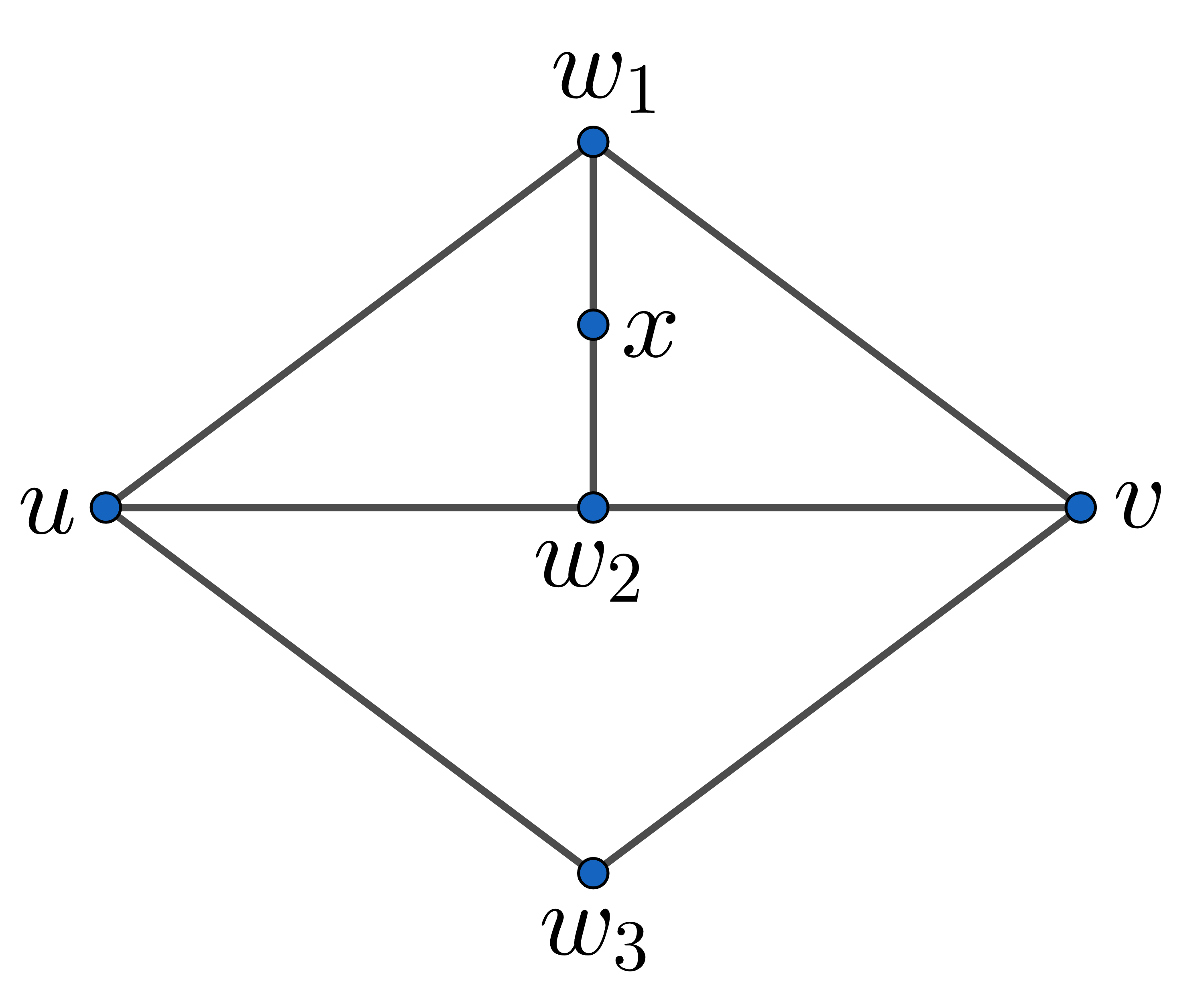}
		\caption{}
		\label{fig:no2a}
	\end{subfigure}
	\hspace{-1.cm}
	\begin{subfigure}{0.48\textwidth}
		\centering
		\includegraphics[width=3.75cm]{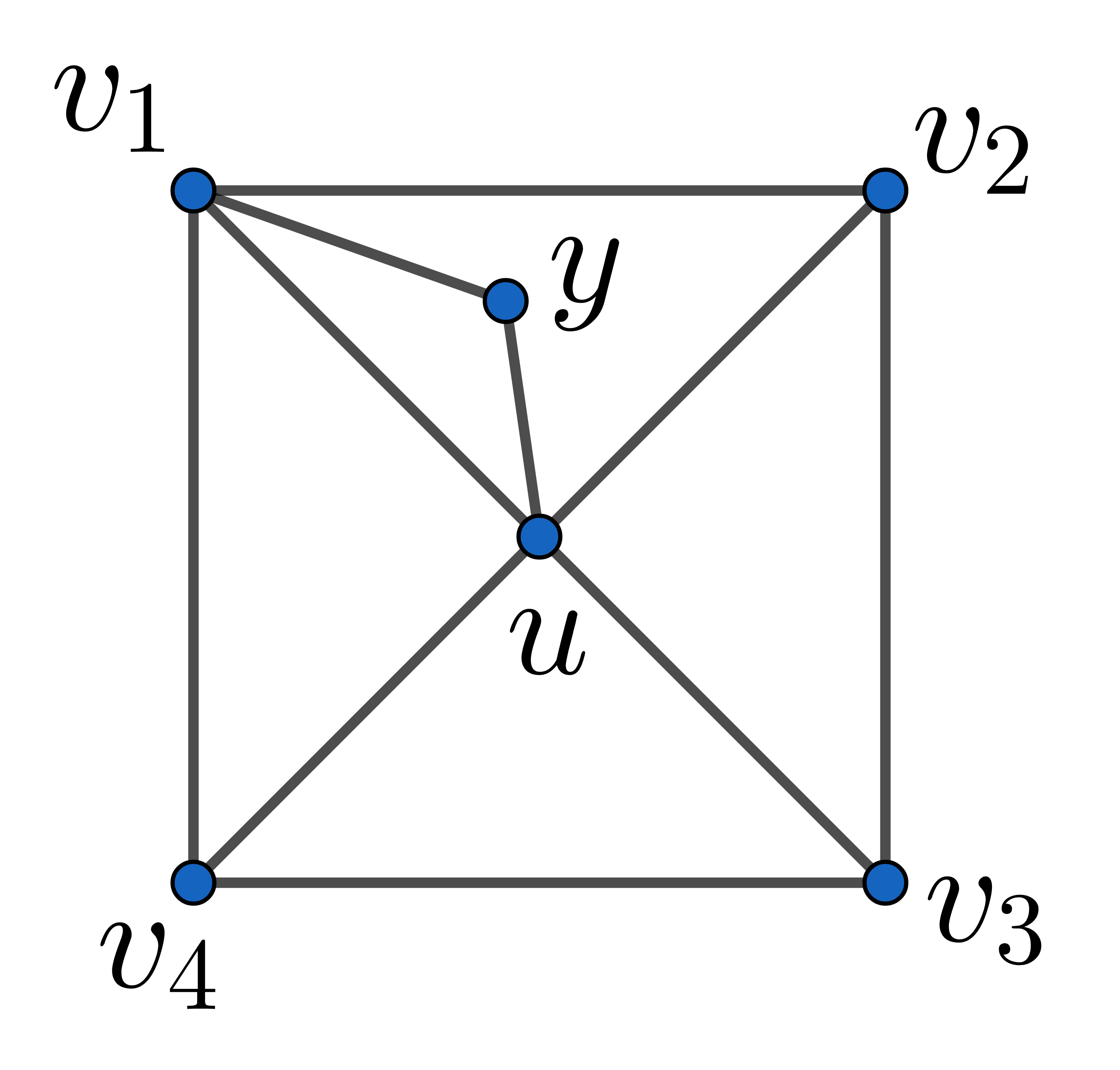}
		\caption{}
		\label{fig:no2b}
	\end{subfigure}
	\caption{Lemma \ref{le:yes2}.}
	\label{fig:no2}
\end{figure}

Take a subgraph of $G$ isomorphic to the square pyramid, with apex $u$ and base $[v_1,v_2,v_3,v_4]$, as in Figure \ref{fig:no2b}. Now $u,v_1$ have a common neighbour $y\not\in\{v_2,v_4\}$, that we may draw w.l.o.g.\ inside of the cycle $u,v_1,v_2$. Next, $y,v_4$ have a common neighbour other than $u,v_1$, and by planarity this common neighbour is $v_2$. Finally, $y,v_3$ have a common neighbour other than $u,v_2$, and by planarity this is impossible.
\end{proof}

Note that in Lemma \ref{le:yes2}, planarity is an essential condition e.g., for $n\geq 2$ the complete graphs satisfy $A(K_n)=\{n-2\}$.

\begin{lemma}
\label{le:2}
Let $G$ be a graph. If $G$ contains a subgraph isomorphic to $K(2,n)$, $n\geq 2$, then there exists $a\geq n$ such that $a\in A(G)$.
\end{lemma}
\begin{proof}
Since $G$ contains a subgraph isomorphic to $K(2,n)$, we may find $u,v,w_1,w_2,\dots,w_n\in V(G)$ such that $N(u,v)\supseteq\{w_1,w_2,\dots,w_n\}$. Then $a:=|N(u,v)|\geq n$, as claimed.
\end{proof}

Combining Lemmas \ref{le:yes2} and \ref{le:2} we deduce the following.
\begin{cor}
	\label{cor:2}
	Let $G$ be a planar graph. We have $2\in A$ if and only if $G$ contains a $4$-cycle.
\end{cor}
\begin{proof}
If $2\in A$, then we may find $u,v,w_1,w_2\in V(G)$ such that $N(u,v)\supseteq\{w_1,w_2\}$, thus $G$ contains the $4$-cycle $u,w_1,v,w_2$. Conversely, if $G$ contains a $4$-cycle, then there exists $n\geq 2$ such that $G$ contains a subgraph isomorphic to $K(2,n)$. By Lemma \ref{le:2}, there exists $a\geq n\geq 2$ such that $a\in A(G)$. By Lemma \ref{le:yes2}, $2\in A(G)$.
\end{proof}

\begin{lemma}
	\label{le:0}
Let $G$ be a connected graph. If $\diam(G)\geq 3$, then $0\in A$.
\end{lemma}	
\begin{proof}
If $\diam(G)\geq 3$, then there exist vertices $u,v$ at distance $3$, thus $N(u,v)=\emptyset$.
\end{proof}


\begin{lemma}
	\label{le:no0}
	Let $G$ be a planar, connected graph on $p$ vertices with $\rad(G)\neq 1$ and $0\not\in A$. Then
	\[p\leq 4M^2+3M+2,\]
	where $M=\max\{a: a\in A\}$.
\end{lemma}
\begin{proof}
	Since $0\not\in A$, by Lemma \ref{le:0} every pair of vertices is at distance at most $2$, hence $\rad(G)=\diam(G)=2$. By \cite[Theorem 2]{goddard2002domination}, in a planar graph of diameter $2$ different from the one in Figure \ref{fig:exc}, there exist two vertices such that every other vertex is adjacent to at least one of them (i.e., its domination number is at most $2$). Hence we may take $x,y\in V(G)$ such that every other element in $V(G)$ is adjacent to at least one of $x,y$.
	\begin{figure}[ht]
		\centering
		\includegraphics[width=2.5cm]{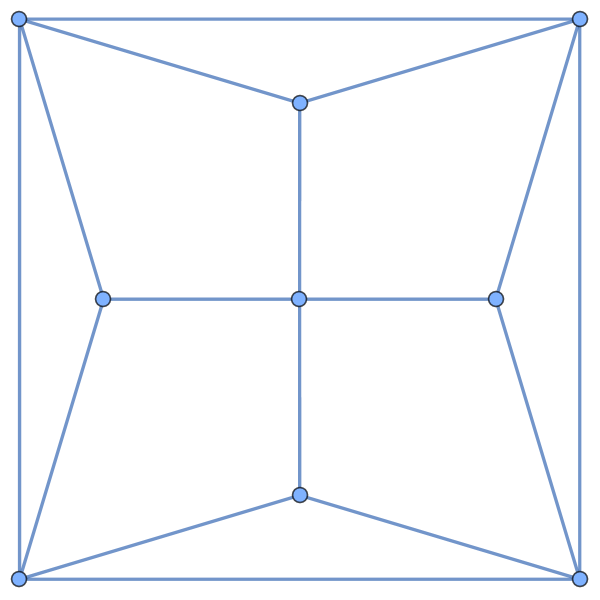}
		\caption{The only planar graph with diameter $2$ and domination number $3$.}
		\label{fig:exc}
	\end{figure}
	
	We write the disjoint union
	\[V(G)=\{x,y\}\cup N(x,y)\cup V_x\cup V_y,\]
	where $V_x$ is the set of vertices $\neq x,y$ adjacent to $x$ but not $y$, and $V_y$ the set of vertices $\neq x,y$ adjacent to $y$ but not $x$. As $\rad(G)\neq 1$, then $V_x,V_y\neq\emptyset$. We wish to give an upper bound for $V_y$.
	
	Since $M=\max\{a: a\in A\}$, each element of $V_x$ and $V_y$ is adjacent to at most $M$ elements of $N(x)$ and at most $M$ elements of $N(y)$, hence in particular its degree in $G$ is at most $2M+1$. For $u\in V(G)$ and $i\geq 1$, we denote by $N_i(u)$ the set of vertices at distance $i$ from $u$ (in particular, $N_1(u)=N(u)$). Letting $z\in V_x$, since $\diam(G)=2$ we have the disjoint union
	\[V(G)=\{z\}\cup N_1(z)\cup N_2(z).\] Since for every $u\in V(G)$ one has
	\[|N_1(u)\cap V_y|\leq|N(u,y)|\leq M,\]
	it follows that
	\[|V_y|=|N_1(z)\cap V_y|+|N_2(z)\cap V_y|\leq M+\sum_{u\in N_1(z)\setminus \{x\}}|N_1(u)\cap V_y|\leq M+2M\cdot M=2M^2+M.\]
	Similarly, $|V_x|\leq 2M^2+M$. Therefore,
	\[p=2+|N(x,y)|+|V_x|+|V_y|\leq 2+M+(2M^2+M)+(2M^2+M)=4M^2+3M+2,\]
	as claimed.
\end{proof}

\begin{lemma}
	\label{le:1}
	If $G$ is a polyhedron, then $A\neq\{1\}$.
\end{lemma}
\begin{proof}
	By contradiction, assume that $G$ is a $(p,q)$-polyhedral graph satisfying $A=\{1\}$. Let $u\in V(G)$ have neighbours $N(u)=\{v_1,v_2,\dots,v_n\}$. Now $u$ has a unique common neighbour with $v_1$, and this neighbour must be one of $v_2,\dots,v_n$, say $v_2$. Thus $u,v_1,v_2$ form a triangle. Similarly, $u,v_3,v_4$ form a triangle and so forth. Therefore, the plane neighbourhood $\Gamma_u(G)$ \eqref{eq:plne} of each vertex $u$ includes $\deg_G(u)/2$ triangles with the common vertex $u$, and $\deg_G(u)/2$ non-triangular faces containing $u$. In particular, every vertex has even degree.
	
	We claim that these triangles are all facial. Indeed, by contradiction, let $x$ be inside the cycle $u,v_1,v_2$. Now $x$ cannot be adjacent to two or more of $u,v_1,v_2$, else $G$ would contain a $4$-cycle contradicting Lemma \ref{le:2}. On the other hand, $|N(x,v_3)|=1$, while $v_1,v_2$ are not adjacent to $v_3$, hence by planarity $N(x,v_3)=\{u\}$. Thereby, none of the vertices inside the cycle $u,v_1,v_2$ are adjacent to $v_1,v_2$, contradicting $3$-connectivity of $G$. Hence indeed for each vertex $u$ of $G$, its neighbourhood includes $\deg_G(u)/2$ facial triangles with the common vertex $u$.
	
	In particular, for each vertex exactly half of the faces that it lies on are triangular. It follows that $f_3=f/2$, where $f$ is the total number of faces in $G$ and $f_i$ the number of faces of length $i$ for $i\geq 3$. Now by the handshaking lemma applied to the dual of $G$, and since $G$ does not contain any $4$-cycles, one has
	\begin{equation}
		\label{eq:1a}
		2q=\sum_{i\geq 3}if_i=3f_3+\sum_{i\geq 5}if_i\geq \frac{3f}{2}+5\sum_{i\geq 5}f_i=\frac{3f}{2}+5(f-f_3)=\frac{3f}{2}+\frac{5f}{2}=4f.
	\end{equation}
	
	On the other hand, since each vertex of $G$ has even degree, by the handshaking lemma applied to $G$ we obtain $2q\geq 4p$, and combining with \eqref{eq:1a} and Euler's formula for planar graphs one has
	\[4p+4r-8=4q=2q+2q\geq 4p+4r,\]
	contradiction.
\end{proof}
The statement of Lemma \ref{le:1} does not hold for planar, connected graphs in general. For instance, a graph given by any number of triangles with a common vertex satisfies $A=\{1\}$.

We are ready to characterise polyhedra satisfying $A=\{0,1\}$.
\begin{cor}
	\label{cor:01}
	Let $G$ be a polyhedron. Then we have $A=\{0,1\}$ if and only if $G$ does not contain any $4$-cycles.
\end{cor}
\begin{proof}
	Let $A=\{0,1\}$. Then $2\not\in A$, so that by Corollary \ref{cor:2}, $G$ does not contain any $4$-cycles. Conversely, suppose that $G$ does not contain any $4$-cycles. By Corollary \ref{cor:2}, $2\not\in A$. By Lemma \ref{le:2}, $A\subseteq\{0,1\}$. By Lemma \ref{le:1}, in fact $A=\{0,1\}$.
\end{proof}

\section{Polyhedra of graph radius $1$}
\label{sec:rad1}

In the present and in the following section, we will begin to populate Table \ref{t:2}. We will need a few auxiliary results about polyhedra of graph radius $1$, that play a central role in our arguments.

\begin{lemma}
	\label{le:ecc}
	Given a polyhedron $G\not\simeq T_\ell$ of graph radius $1$, there is in $G$ a unique vertex of eccentricity $1$.
\end{lemma}
\begin{proof}	
	Denote by $i$ the number of vertices of eccentricity $1$ in a graph $G$. By definition a vertex of eccentricity $1$ is adjacent to all other vertices. If $i\geq 4$, since by planarity $K_5$ is not a subgraph of $G$, then $i=4$ and $|V(G)|-i=0$ so that $G$ is the tetrahedron $T_2$. If $i=3$, since by planarity $K(3,3)$ is not a subgraph of $G$, then $|V(G)|-i\leq 2$ so that $G$ is the triangular bipyramid $T_3$.
	
	Let $i=2$. We use the $3$-connectivity of $G$ to see that $H:=G-u-v$ is connected, where $u,v$ are the vertices of $G$ with eccentricity $1$. If $G$ is a $(p,q)$-graph, then $H$ has $p-2$ vertices and
	\[q-(p-1)-(p-2)=q+3-2p\]
	edges. On one hand, since $H$ is connected one has
	\[q+3-2p\geq (p-2)-1,\]
	whence $q\geq 3p-6$. On the other hand, since $G$ is planar, we have $q\leq 3p-6$, so that ultimately $G$ is a triangulation, and $H$ is a connected $(p-2,p-3)$-graph i.e., it is a tree. Furthermore, no vertex of $H$ may have degree $3$ in $H$, or $G$ would contain a subgraph isomorphic to $K(3,3)$. Hence $H$ is a path, thus $G\simeq T_\ell$ for some $\ell\geq 4$.
\end{proof}

Given a graph $G$ and $u\in V(G)$, recall the definition \eqref{eq:plne} of plane neighbourhood $\Gamma_u(G)$.
\begin{lemma}
	\label{le:33}
If $G$ is a polyhedron and $u$ a vertex of eccentricity $1$, then $\Gamma_u(G)$ is a pyramid. Moreover, $G$ has at least two vertices of degree $3$.
\end{lemma}
\begin{proof}
Let $u\in V(G)$ be a vertex of eccentricity $1$. We take
\[G'=G-u.\]
Since $G$ is a polyhedron, then $G'$ is planar and $2$-connected, and moreover by construction there exists a planar embedding of $G'$ such that all the vertices are on the external face. That is to say, $G'$ is outerplanar and $2$-connected. By \cite[Lemma 3.3]{maffucci2025regularity}, $G'$ is Hamiltonian. Therefore, the plane neighbourhood
\[\Gamma_u(G)\]
is a $|V(G')|$-gonal pyramid, as claimed.

Let
\[[v_1,v_2,\dots,v_n], \qquad n=|V(G')|\]
be the base of $\Gamma_u(G)$. To construct $G$ starting from $\Gamma_u(G)$, one inserts the remaining edges one at a time. One starts with a given $v_iv_j$, where
\[1\leq i<j\leq n \qquad\text{ and }\qquad 2\leq j-i\leq n-2.\]
In the graph $\Gamma_u(G)+v_iv_j$, there is a vertex of degree $3$ among
\begin{equation}
	\label{eq:am1}
v_{i+1},v_{i+2},\dots,v_{j-1},
\end{equation}
and another among
\begin{equation}
	\label{eq:am2}
v_{1},v_{2},\dots,v_{i-1},v_{j+1},v_{j+2},\dots,v_{n}.
\end{equation}
By planarity, as one inserts the remaining edges, at each step there are always at least two vertices of degree $3$.
\end{proof}

Another way of phrasing Lemma \ref{le:33} is, a polyhedron of graph radius $1$ may be obtained from a pyramid by adding edges.

\begin{lemma}
	\label{le:nrad1}
Let $G$ be a polyhedron of graph radius $1$, $u\in V(G)$ of eccentricity $1$, and
\[[v_1,v_2,\dots,v_n]\]
the base of the pyramid $\Gamma_u(G)$ as per Lemma \ref{le:33}. Then for every $1\leq i\neq j\leq n$ one has
\[|N(v_i,u)|=\deg_G(v_i)-1\geq 2\]
and
\[1\leq |N(v_i,v_j)|\leq 3,\]
with $|N(v_i,v_j)|=3$ if and only if $v_i,v_j$ are non-consecutive vertices along a $4$-cycle in $G-u$.
\end{lemma}
\begin{proof}
To check the first statement, we note that, since $u$ is of eccentricity $1$,
\[N(v_i,u)=\{w\in V(G-u) : v_iw\in E(G)\}.\]
For the second statement, the lower bound is clear from $N(v_i,v_j)\supseteq\{u\}$. For the upper bound, by planarity $v_i,v_j$ may have at most one common neighbour among \eqref{eq:am1} and at most one common neighbour among \eqref{eq:am2}. We have $|N(v_i,v_j)|=3$ if and only if $v_i,v_j$ have a common neighbour $x$ among \eqref{eq:am1} and a common neighbour $y$ among \eqref{eq:am2}, if and only if $v_i,x,v_j,y$ is a $4$-cycle in $G-u$.
\end{proof}

\begin{prop}
	\label{prop:arad1}
Let $G\not\in\ks_1$ be a polyhedron of graph radius $1$, and $u\in V(G)$ of eccentricity $1$. Then we have
\[A(G)=\{1,2\}\cup\{\deg_G(v)-1 : v\in V(G-u)\}\]
if and only if there are no $4$-cycles in $G-u$, otherwise we have
\[A(G)=\{1,2,3\}\cup\{\deg_G(v)-1 : v\in V(G-u)\}.\]
\end{prop}
\begin{proof}
By Theorem \ref{thm:1}, since $G\not\in\ks_1$, one has $1\in A$. By Lemma \ref{le:33}, $G$ has at least two vertices of degree $3$. Hence by Lemma \ref{le:nrad1}, $2\in A$. Also by Lemma \ref{le:nrad1}, if $a\in A$ then
\[a\in\{1,2,3\}\cup\{\deg_G(v)-1 : v\neq u\}.\]
Still by Lemma \ref{le:nrad1}, if $G-u$ contains a $4$-cycle, then $3\in A$.
\end{proof}


We are ready to classify polyhedra of graph radius $1$ satisfying $A=\{1,2,\ell\}$ for $\ell\geq 2$. We begin with the case $\ell\leq 4$.
\begin{prop}
	\label{prop:rad1}
Let $G$ be a polyhedron of graph radius $1$ on at least six vertices. Then
\begin{enumerate}
\item
$A=\{1,2\}$ if and only if $G$ is an $n$-gonal pyramid for some $n\geq 5$;
\item
$A=\{1,2,3\}$ if and only if $G\in\cw_3$;
\item
$A=\{1,2,4\}$ if and only if $G\in\cw_4$.
\end{enumerate}
\end{prop}
\begin{proof}
By Lemma \ref{le:33}, $G$ is obtained by adding edges to the pyramid $H:=\Gamma_u(G)$ of apex $u$ and base
\[B=[v_1,v_2,\dots,v_n], \qquad n\geq 5.\]
Along $B$ we consider indices modulo $n$ i.e., we take $v_0:=v_n$ and for every $i\in\mathbb{Z}$, $v_i:=v_{(i \mod n)}$.
\begin{enumerate}
\item
If $A=\{1,2\}$, then by Proposition \ref{prop:arad1} each vertex on the base of $H$ has degree $3$ in $G$, thus $G=H$. Conversely, it is easy to see that if $G=H$, then for every $i<j$ one has
\[N(v_i,v_j)=
\begin{cases}
2&j-i\in\{2,n-2\}\\
1&\text{otherwise}.
\end{cases}\]
\item
If $A=\{1,2,3\}$, then by Proposition \ref{prop:arad1} each vertex on $B$ has degree $3$ or $4$ in $G$. Hence $G$ is obtained from $H$ by adding pairwise independent edges. Since pyramids satisfy $A=\{1,2\}$, then we need to add at least one edge. Thus $G\in\cw_3$.

Conversely, if $G\in\cw_3$ i.e., $G$ is obtained from $H$ by adding pairwise independent edges, then each vertex on $B$ has degree either $3$ or $4$ in $G$. Since we have added at least one edge, then by Lemmas \ref{le:33} and \ref{le:nrad1} we have $\{2,3\}\subseteq A\subseteq\{1,2,3\}$. Now let $v_iv_j\in E(G)$ with $i<j$, as in Figure \ref{fig:w3}. Since $G$ is obtained from $H$ by adding independent edges, we deduce that
\[N(v_i)=\{u,v_{i-1},v_{i+1},v_j\}.\]
By planarity, $v_{i+1}v_{i-1}\not\in E(G)$. If $j\neq i+2$, again since $G$ is obtained from $H$ by adding independent edges, we have $v_{i+1}v_j\not\in E(G)$. It follows that $N(v_{i},v_{i+1})=\{u\}$, thus $1\in A$. If instead $j=i+2$, then since $n\geq 5$ we obtain $v_{i-1}v_j\not\in E(G)$, so that $N(v_{i-1},v_{i})=\{u\}$, thus $1\in A$. Ultimately in this case $A=\{1,2,3\}$ as required.
\item
If $A=\{1,2,4\}$, then by Proposition \ref{prop:arad1} each vertex on $B$ has degree $3$ or $5$ in $G$. Therefore, the subgraph of $G$ generated by the edges $E(G)\setminus E(H)$ is $2$-regular i.e., it is a (non-empty) disjoint union of cycles. By contradiction, let
\[w_1,w_2,w_3,w_4\]
be one of these cycles. Then
\[N(w_1,w_3)=\{u,w_2,w_4\},\]
thus $3\in A$, contradiction. Now let
\[v_i,v_j,v_k, \qquad 1\leq i<j<k\leq n, \quad |j-i|, |k-j|, |n-(k-i)|\geq 2\]
be one of the cycles. If $j=i+2$, then
\[N(v_i,v_{i+2})=\{u,v_k,v_{i+1}\},\]
thus $3\in A$, impossible. If $j=i+3$, then
\[N(v_i,v_{i+2})=\{u,v_{i+1},v_{i+3}\},\]
thus $3\in A$, impossible. Hence in fact
\[|j-i|, |k-j|, |n-(k-i)|\geq 4.\]
Now by contradiction assume that $v_{i}v_j,v_{i+1}v_{j-1}\in E(G)$ for some $i<j$. Then
\[N(v_i,v_{j-1})=\{u,v_{i+1},v_{j}\},\]
thus $3\in A$, contradiction.

Conversely, let $G\in\cw_4$. The elements of $V(H)$ have degree either $3$ or $5$ in $G$. Thus by Lemma \ref{le:nrad1}, one has for every $1\leq i<j\leq n$
\[|N(v_i,u)|=\deg_G(v_i)-1\in\{2,4\}\]
and
\[1\leq |N(v_i,v_j)|\leq 3.\]
It remains to show that $1\in A$ and $3\not\in A$. Let $\cc$ be one of the cycles added to $H$, and $v_{i}\in\cc$, $4\leq i\leq n-1$. If the vertex on $\cc$ preceding $v_i$ is not $v_{i-2}$, then by planarity
\[N(v_{i+1},v_{i-2})=\{u\},\]
thus $1\in A$. If the vertex on $\cc$ preceding $v_i$ is $v_{i-2}$, then $v_{i-3}\not\in\cc$, hence by planarity
\[N(v_{i+1},v_{i-3})=\{u\},\]
thus $1\in A$.

Now by contradiction, let there exist $1\leq i<j\leq n$ such that
\[|N(v_i,v_j)|=3.\]
That is to say, $j\geq i+2$ and we may find $h$ among \eqref{eq:am1} and $k$ among \eqref{eq:am2} such that
\[N(v_i,v_j)=\{u,v_h,v_k\}.\]
Since no cycle added to $H$ is of length $4$, then at least one of $h\in\{i+1,j-1\}$ or $k\in\{i-1,j+1\}$ holds, w.l.o.g.\ $h=i+1$ (Figure \ref{fig:w4}). Hence one of the cycles added to $H$ contains $v_i,v_{i+1}$, contradiction. We have proven that $3\not\in A$.
\begin{figure}[ht]
	\centering
	\includegraphics[width=4.cm]{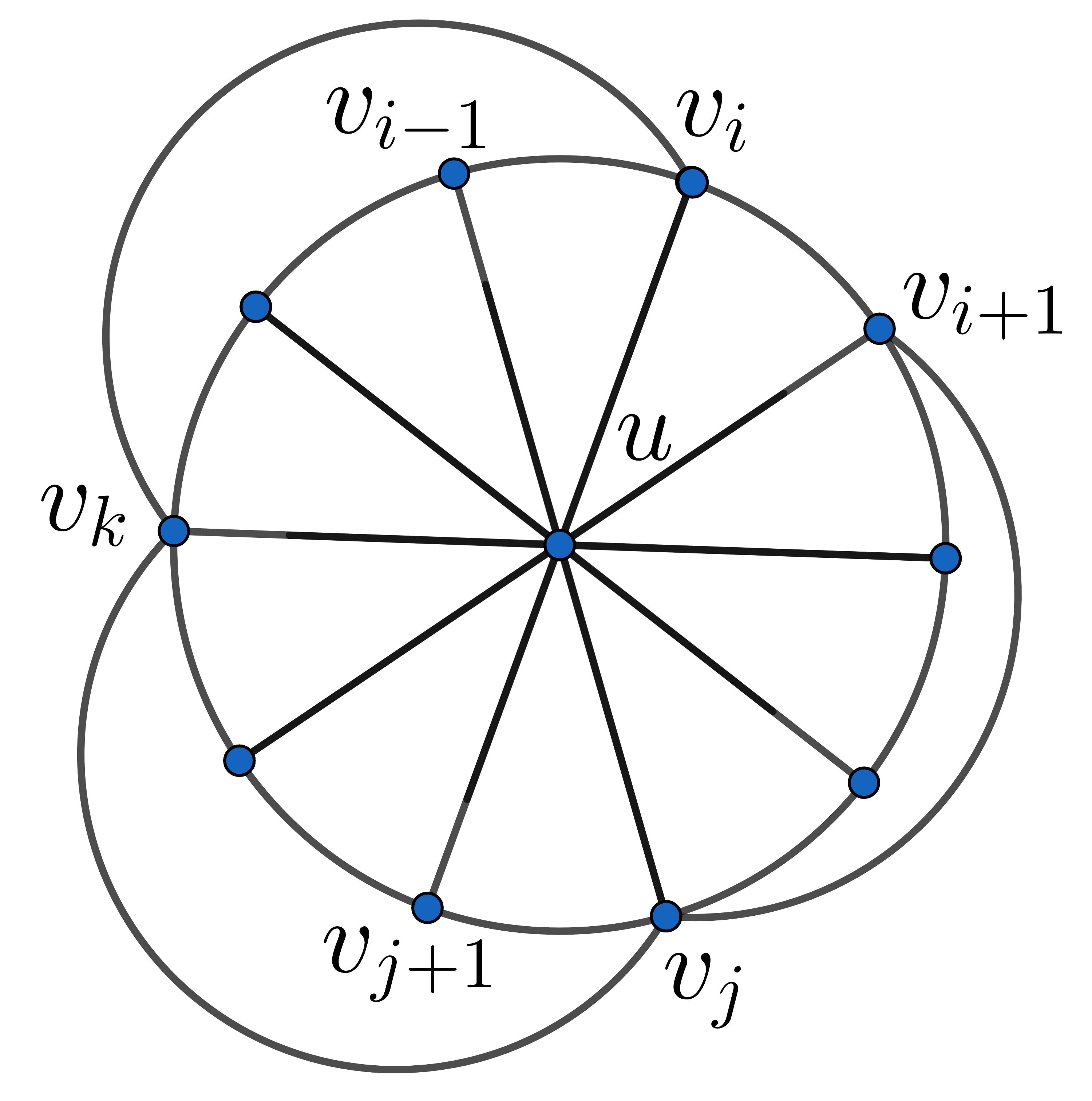}
	\caption{If $G\in\cw_4$, then $3\not\in A$.}
	\label{fig:w4}
\end{figure}
\end{enumerate}
\end{proof}

In fact, we are ready to classify all polyhedra such that $A=\{1,2,\ell\}$ for $2\leq\ell\leq 4$.
\begin{cor}
	\label{cor:1234}
	If $G$ is a polyhedron satisfying $A=\{1,2\}$, then either $G$ is a pyramid or $p\leq 24$. If $G$ is a polyhedron satisfying $A=\{1,2,3\}$, then either $G\in\cw_3$ or $p\leq 47$. If $G$ is a polyhedron satisfying $A=\{1,2,4\}$, then either $G\in\cw_4$ or $p\leq 78$.
\end{cor}
\begin{proof}
	In each case, since $0\not\in A$, by Lemma \ref{le:0} we have $\diam(G)\leq 2$, hence $\rad(G)\in\{1,2\}$. The results of the present Corollary now follow from Lemma \ref{le:no0} and Proposition \ref{prop:rad1}.
\end{proof}

We may now characterise polyhedra of order at least $25$ satisfying $A=\{0,1,2\}$.
\begin{cor}
	\label{cor:012}
	Let $G$ be a polyhedron on at least $25$ vertices. Then we have $A=\{0,1,2\}$ if and only if $G$ is not a pyramid, contains a $4$-cycle, but no subgraph isomorphic to $K(2,3)$.
\end{cor}
\begin{proof}
	Let $G$ be a polyhedron of order at least $25$. Suppose that $A=\{0,1,2\}$. By Lemma \ref{le:2}, $G$ does not contain any copies of $K(2,3)$. By Corollary \ref{cor:2}, $G$ contains a $4$-cycle. By Proposition \ref{prop:rad1}, first part, pyramids on at least $6$ vertices verify $A=\{1,2\}$, thus $G$ is not a pyramid.
	
	Vice versa, suppose that $G$ is not a pyramid, contains a $4$-cycle, but no subgraph isomorphic to $K(2,3)$. As $G$ contains a $4$-cycle, by Corollary \ref{cor:2} $2\in A$. Since $G$ does not contain a copy of $K(2,3)$, then $A$ cannot contain any number greater than $2$. It follows that
	\[\{2\}\subseteq A\subseteq\{0,1,2\}.\]
	By Theorem \ref{thm:1}, the only polyhedra where $A$ does not contain $1$ nor any number greater than $2$ are the tetrahedron, cube, and icosahedron. Since $G$ has at least $25$ vertices, we may exclude these, thus $1\in A$. By Corollary \ref{cor:1234},
	since $G$ has at least $25$ vertices and is not a pyramid, we exclude the case $A=\{1,2\}$. This leaves $A=\{0,1,2\}$, as desired.
\end{proof}

To complete the classification of polyhedra of graph radius $1$ satisfying $A=\{1,2,\ell\}$ for $\ell\geq 2$, we will show that in case $\ell\geq 5$ there are no solutions. In fact, the following result is more general, and will be useful in Section \ref{sec:compl}.
\begin{lemma}
	\label{le:3o4r1}
Let $G$ be a polyhedron of graph radius $1$, and $a\in A$ such that $a\geq 5$. Then either $3$ or $4$ belong to $A$.
\end{lemma}
\begin{proof}
By contradiction, assume that $G$ is a polyhedron of graph radius $1$ with $3,4\not\in A$, and let $a$ be the smallest number such that $a\geq 5$ and $a\in A$. The graphs $T_\ell$, $\ell\geq 3$ verify $A=\{2,3,\ell\}$, thus $3\in A$. 

Henceforth we may assume that $G\not\simeq T_\ell$, so that by Lemma \ref{le:ecc} there exists in $G$ a unique vertex $u$ of eccentricity $1$. By Lemma \ref{le:33}, $\Gamma_u(G)$ is a pyramid, and by Proposition \ref{prop:arad1}, the elements of $V(G-u)$ have degree $3$ or $\geq a+1$ in $G$. Thus the subgraph of $G$ generated by the elements of $V(G-u)$ of degree $\geq a+1$ in $G$ is an outerplanar graph of minimum degree $a-2\geq 3$.

On the other hand, given a connected component $H$ of an outerplanar graph, we consider an endblock $H'$ of $H$. Then $H'$ is a $2$-connected, outerplanar graph i.e., a polygon with some added diagonals, and as such, it contains at least two vertices of degree $2$ \cite[Lemma 3.3]{maffucci2025regularity}. Since $H'$ is an endblock of $H$, at least one of these vertices has degree $2$ in $H$ as well. We have reached a contradiction.
\end{proof}

Note that Lemma \ref{le:3o4r1} does not hold in general for planar, connected graphs of radius $1$. For instance, starting from the graph consisting of $n\geq 4$ triangles with the common vertex $u$, we add edges between a vertex $\neq u$ of one of the triangles and one vertex $\neq u$ of $m$ of the other triangles, where $4\leq m\leq n$. Then the resulting planar, connected graph of radius $1$ satisfies $A=\{1,2,m+1\}$.

\section{Completing the classification of types of polyhedra}
\label{sec:compl}
The goal of the present section is to prove that the list of types of polyhedra with $1\in A$ in the first column of Table \ref{t:2} is complete. To achieve this, we require one last key result, classifying all polyhedra such that $0,3,4\not\in A$ and $\ell\in A$ for some $\ell\geq 5$.
\begin{prop}
	\label{prop:034}
Let $G$ be a polyhedral graph. We have $0,3,4\not\in A$ and $\ell\in A$ such that $\ell\geq 5$ if and only if either $G\simeq B_\ell$ or $G\simeq B'_\ell$, where $\ell$ is even.
\end{prop}
\begin{proof}
Let $\ell$ be the smallest number satisfying $\ell\geq 5$ and $\ell\in A$, and $x,y\in V(G)$ such that $|N(x,y)|=\ell$. Call
\begin{equation}
	\label{eq:wi}
	w_1,w_2,\dots,w_\ell
\end{equation}
the common neighbours of $x,y$, appearing in this cyclic order around $x$ in the planar immersion of $G$, as in Figure \ref{fig:diam2a}. We consider the indices modulo $\ell$, and take $w_0:=w_\ell$. We claim that every vertex $\neq x,y$ of $G$ is adjacent to at least one of $x,y$, so that $\{x,y\}$ is also the dominating set prescribed in Lemma \ref{le:no0}. Indeed, by contradiction let $z\not\in N(x),N(y)$ lie inside the cycle
\[u,w_i,v,w_{i+1}.\]
Then by planarity $N(z,w_{i+3})=\emptyset$, contradiction.
\begin{figure}[ht]
	\centering
	\begin{subfigure}{0.48\textwidth}
		\centering
		\includegraphics[width=7cm]{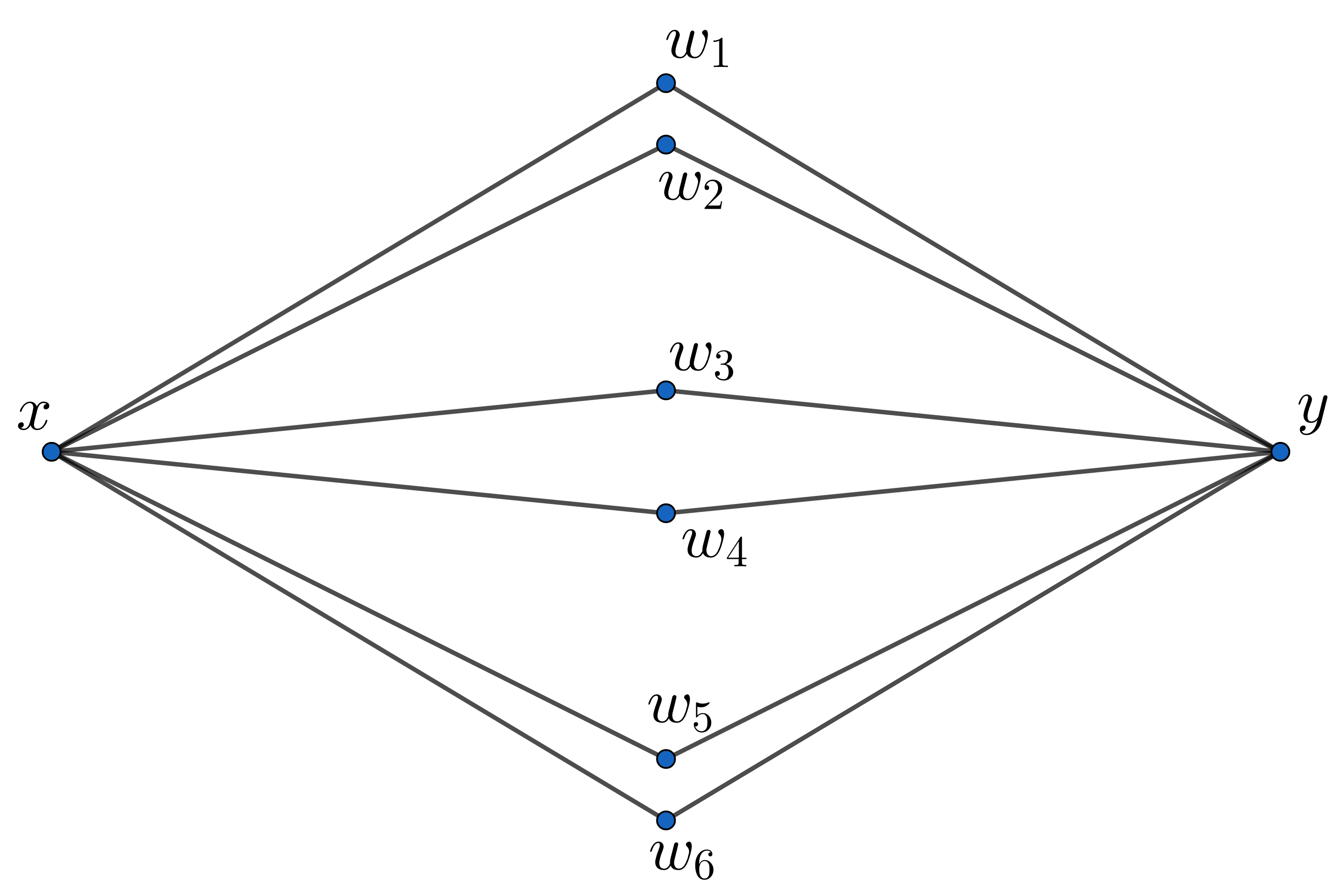}
		\caption{}
		\label{fig:diam2a}
	\end{subfigure}
	\hfill
	\begin{subfigure}{0.48\textwidth}
		\centering
		\includegraphics[width=7cm]{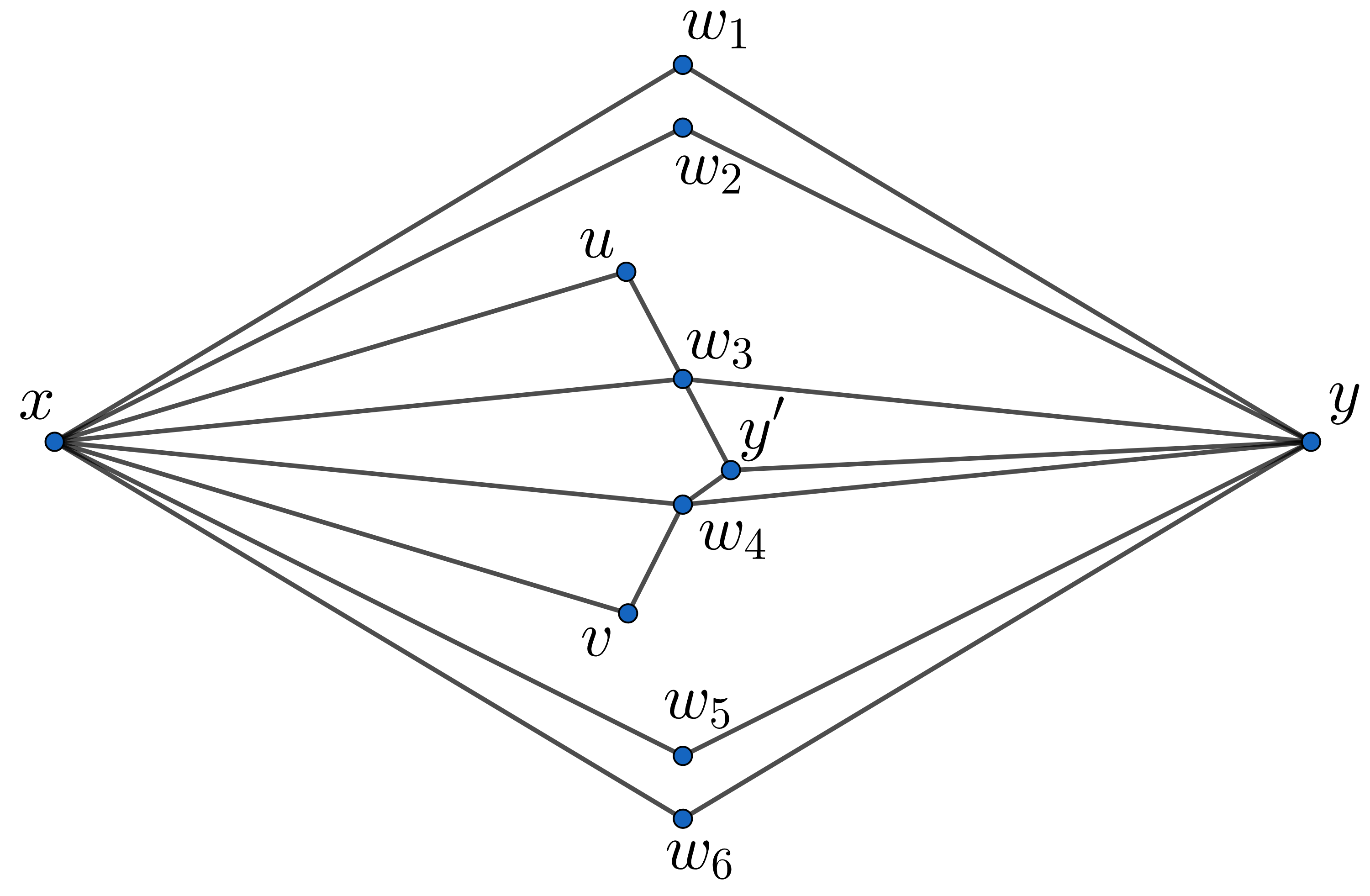}
		\caption{}
		\label{fig:diam2b}
	\end{subfigure}
	\caption{Proposition \ref{prop:034}.}
	\label{fig:diam2}
\end{figure}

As in Lemma \ref{le:no0}, call $V_x$ the subset of vertices $\neq x,y$ adjacent to $x$ but not $y$, and $V_y$ the subset of vertices $\neq x,y$ adjacent to $y$ but not $x$. As $G$ is $3$-connected, then $G-u-v$ is connected, thus there exists a path $P$ in $G-u-v$ containing all of \eqref{eq:wi}. By planarity, w.l.o.g.\ we may assume that \eqref{eq:wi} appear along $P$ in this order.

Along $P$, for every $i\geq 1$ the vertices $w_i,w_{i+1},w_{i+2}$ cannot be consecutive, otherwise 
\[
N(w_i,w_{i+2})\supseteq\{u,v,w_{i+1}\}
\]
thus by planarity $N(w_i,w_{i+2})=\{u,v,w_{i+1}\}$, impossible. Hence along $P$ either
\[w_1,u,w_2,w_3,v,w_4,\dots,w_\ell\]
appear in this order, or
\[w_1,w_2,u,w_3,w_4,v,w_5,\dots,w_\ell\]
appear in this order. In either case, $N(u,v)$ is either $\{x\}$ or $\{y\}$, say the former.

We claim that $V_y=\emptyset$. By contradiction, let $y'\in V_y$. Since $N(y',u)$ and $N(y',v)$ are not empty, then 
there exists $i\in\{1,2\}$ such that
\[w_i,u,w_{i+1},y',w_{i+2},v,w_{i+3}\]
appear in this order along $P$, with
\[uw_{i+1}, \ w_{i+1}y', \ y'w_{i+2}, \ w_{i+2}v\in E(G),\]
as in Figure \ref{fig:diam2b}. Hence
\[N(w_{i+1},w_{i+2})\supseteq\{x,y,y'\},\]
so that $|N(w_{i+1},w_{i+2})|\geq\ell\geq 5$. Therefore, there are at least two other vertices adjacent to $w_{i+1},w_{i+2}$ and to one of $x,y$. However, by planarity this cannot happen. Hence indeed $V_y=\emptyset$.

Since $V_y=\emptyset$, we deduce that $xy\not\in E(G)$. Indeed, assuming $xy\in E(G)$ we have $\ecc(x)=1$, so that by Lemma \ref{le:3o4r1} either $3$ or $4$ belongs to $A$, contradiction.

Next, we claim that for $1\leq i\leq\ell-1$ along $P$ there are either $0$ or $2$ vertices between $w_i$ and $w_{i+1}$. For each $i\geq 1$, we have $N(w_i,w_{i+1})=\{x,y\}$, since $|N(w_i,w_{i+1})|\geq\ell$ is impossible by planarity, as there would need to be at least three vertices adjacent to all of $x,w_i,w_{i+1}$. It follows that for $i\geq 1$ there cannot be exactly one vertex between $w_i$ and $w_{i+1}$ along $P$. Now assume that
\[w_i,u_1,u_2,u_3,w_{i+1}\]
appear in order. Since $N(u_2,y)\neq\emptyset$ and $xy\not\in E(G)$, then $u_2$ is adjacent to one of $w_i,w_{i+1}$, say $w_i$. Then
\[N(x,u_2)\supseteq\{w_i,u_1,u_3\},\]
thus $|N(x,u_2)|\geq\ell$. By planarity, any element of $N(x,u_2)$ other than $w_i,u_1,u_3$ lies inside the cycle $x,u_1,u_2,u_3$, and as such, this element has no common neighbours with $y$, contradiction. Therefore indeed for $1\leq i\leq\ell-1$ along $P$ there are either $0$ or $2$ vertices between $w_i$ and $w_{i+1}$.

Next, for every $i\geq 1$ since $N(y,w_i)\neq\emptyset$, then $w_i$ is adjacent to at least one of $w_{i-1}$ and $w_{i+1}$. As we have already seen that these three vertices cannot be consecutive along $P$, then in fact $w_i$ is adjacent to exactly one of $w_{i-1}$ and $w_{i+1}$.

Now if
\[w_i,u_1,u_2,w_{i+1}\]
are consecutive along $P$, then $w_iw_{i+1}\not\in E(G)$, else $N(w_i,u_2)=\{x,u_1,w_{i+1}\}$. It follows that for every $1\leq i\leq\ell-1$, along $P$ between $w_i$ and $w_{i+1}$ there are exactly two vertices if and only if $w_iw_{i+1}\not\in E(G)$, otherwise there are none.

All our considerations so far imply that either
\[P: w_1,w_2,u_1,u_2,w_3,w_4,u_3,u_4,\dots,w_{\ell-1},w_\ell,u_{\ell-1},u_\ell,\]
and moreover $u_\ell w_1\in E(G)$, or
\[P: w_1,w_2,u_1,u_2,w_3,w_4,u_3,u_4,\dots,w_{\ell-1},w_\ell,\]
and moreover $w_\ell w_1\not\in E(G)$. In either case, $\ell$ is even. The graph constructed so far is isomorphic to $B_\ell$ or $B'_\ell$, and it remains to show that in either case there are no more vertices and/or edges in $G$.

If we wish to add an edge without introducing new vertices, then by planarity this edge is of the form $u_iw_i$, $i$ even, or $u_iw_{i+2}$, $i$ odd. If $i$ is even, then $N(w_i,w_{i+1})=\{x,y,u_i\}$, while if $i$ is odd, then $N(w_{i+1},w_{i+2})=\{x,y,u_i\}$, impossible in either case. We now add a new vertex $z$. As we have seen, $zx\in E(G)$, and for $N(z,y)$ not to be empty, $z$ is also adjacent to an element of \eqref{eq:wi}, say $w_2$. Now
\[N(x,w_2)\supseteq\{w_1,u_1,z\},\]
thus $|N(x,w_2)|\geq\ell$, impossible.  
Therefore finally $G$ is isomorphic either to $B_\ell$ or to $B'_\ell$ with $\ell$ even.

Conversely, take an even $\ell\geq 6$. We note that $B_\ell$ and $B'_\ell$ of Definition \ref{def:B} are polyhedra, and $|N(u,v)|=\ell$. Moreover one has
\[|N(b_i,u)|=2, \qquad |N(b_i,v)|=1\]
and
\begin{equation*}
|N(b_i,b_j)|=
\begin{cases}
2&i,j \text{ both congruent to }1,2 \text{ modulo } 4,\text{ or } |j-i|\in\{2,2\ell-2\},\\
1&\text{ otherwise}
\end{cases}
\end{equation*}
for every $1\leq i\neq j\leq 2\ell$ in the case of $B_\ell$, and $1\leq i\neq j\leq 2\ell-2$ in the case of $B'_\ell$.
\end{proof}

We are ready to complete the proof for the classification of all types of polyhedra.
\begin{prop}
	\label{prop:class}
	Let $G$ be a polyhedron such that $1\in A(G)$. Then $A(G)$ appears in the first column of Table \ref{t:2}.
\end{prop}
\begin{proof}
	If $2\not\in A$, then by Lemma \ref{le:yes2} one has $A\subseteq\{0,1\}$. Since $A\neq\{1\}$ by Lemma \ref{le:1}, then $2\not\in A$ implies $A=\{0,1\}$. If $0,1,2\in A$, we obtain the types $\{0,1,2\}$ and $\{0,1,2\}\cup A'_3$.
	
	This leaves the case where $1,2\in A$ and $0\not\in A$. By Proposition \ref{prop:034}, the admissible types here are $\{1,2\}$, $\{1,2,3\}$, $\{1,2,3\}\cup A'_4$, $\{1,2,4\}$, $\{1,2,4\}\cup A'_5$, and $\{1,2,\ell\}$ for even $\ell\geq 6$.
\end{proof}

\section{Wide classes of solutions}
\label{sec:wide}
To complete the proof of Theorem \ref{thm:1}, it remains to construct infinitely many polyhedra for each of the types $\{1,2,3\}\cup A'_4$, $\{1,2,4\}\cup A'_5$, and $\{0,1,2\}\cup A'_3$, where $A'_n$ is a finite, non-empty set of integers $\geq n$.

\begin{prop}
	\label{prop:123}
For every (finite) non-empty set $A'_4$ of integers $\geq 4$, there are infinitely many polyhedra $G$ satisfying
\[A(G)=\{1,2,3\}\cup A'_4.\]
\end{prop}
\begin{proof}
A caterpillar $T$ is a tree graph where deleting the vertices of degree $1$ results in a simple path (the `central path' of $T$). The degree sequence of a caterpillar $T$ is given by \cite[Proof of Lemma 2.1]{maffucci2024characterising}
\[t_1,t_2,\dots,t_m,1^{2+\sum_{i=1}^{m}(t_i-2)}\]
where the exponent means repeated values, $m\geq 1$, and $t_1,t_2,\dots,t_m\geq 2$ are the degrees of the vertices along the central path of $T$. If $m=1$ then in particular $T$ is a star graph. We also define
\[h(T):=6+\sum_{i=1}^{m}(t_i-1).\]

Given
\[A'_4=\{a_1,a_2,\dots,a_m\},\]
we wish to construct a polyhedron $G$ such that $A(G)=\{1,2,3\}\cup A'_4$. Considering the caterpillar $T$ of central vertex degrees
\[t_i=a_i-2, \qquad 1\leq i\leq m,\]
we start with an $h(T)$-gonal pyramid of apex $u$, and add the edges of $T$ to the base of the pyramid. This may be done so that the resulting graph $G$ is still simple and planar, as illustrated in Figure \ref{fig:123}. A similar construction may be found in \cite{maffucci2024characterising}, where the focus is on degree sequences with unique realisation up to isomorphism.
\begin{figure}[ht]
	\centering
	\begin{subfigure}{0.48\textwidth}
		\centering
		\includegraphics[width=4.5cm]{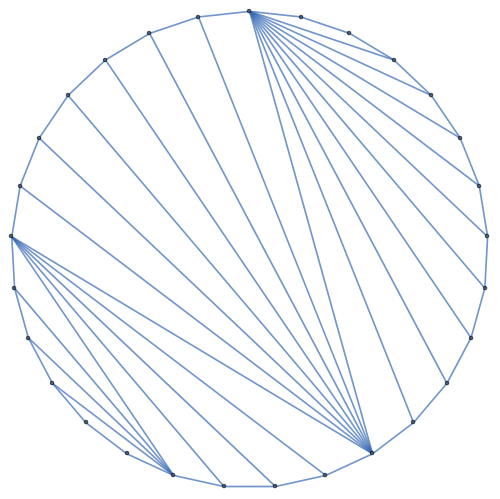}
		\caption{The graph $G-u$ for $A'_4=\{6,7,10,12\}$.}
		\label{fig:123a}
	\end{subfigure}
	\hspace{-1.cm}
	\begin{subfigure}{0.48\textwidth}
		\centering
		\includegraphics[width=3cm]{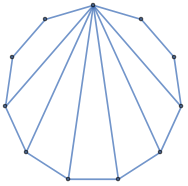}
		\caption{The graph $G-u$ for $A'_4=\{8\}$.}
		\label{fig:123b}
	\end{subfigure}
	\caption{Proposition \ref{prop:123}. For $A'_4=\{6,7,10,12\}$, we consider the caterpillar with central vertices of degrees $4,5,8,10$. If $A'_4=\{8\}$, we take the star with central vertex of degree $6$.}
	\label{fig:123}
\end{figure}

The vertex of $G-u$ corresponding to the vertex of degree $a_i-2$ in $T$ has degree $a_i$ in $G-u$ ($a_i+1$ in $G$) for $1\leq i\leq m$. According to Proposition \ref{prop:arad1}, since $G$ contains vertices of degree $4$, then $A(G)=\{1,2,3\}\cup A'_4$, as desired.

To obtain infinitely many polyhedra of such type we modify the above construction as follows. We consider caterpillars where the central vertices have degrees $t_1,t_2,\dots,t_{m'}$ where $m'\geq m$, each of $t_1,t_2,\dots,t_{m'}$ takes one of the values $a_1-2,a_2-2,\dots,a_m-2$, and each of $a_1-2,a_2-2,\dots,a_m-2$ appears at least once in $\{t_1,t_2,\dots,t_{m'}\}$.
\end{proof}
A possible code for Proposition \ref{prop:123}, in the Mathematica language, may be found in Appendix \ref{app:a}.

\begin{prop}
	\label{prop:124}
For every (finite) non-empty set $A'_5$ of integers $\geq 5$, there are infinitely many polyhedra $G$ satisfying
\[A(G)=\{1,2,4\}\cup A'_5.\]
\end{prop}
\begin{proof}
We begin with a construction similar to Proposition \ref{prop:123}. Given
\[A'_5=\{a_1,a_2,\dots,a_m\},\]
we write $t_i=a_i-2$, $1\leq i\leq m$. If $A'_5$ contains an odd number of odd elements, we also take $t_{m+1}=a_j-2$, where $a_j$ is the smallest odd element of $A'_5$. We then consider the caterpillar $T$ of such central vertex degrees, and add its edges to the $h(T)$-gonal pyramid of apex $u$.

Next, we insert extra edges between vertices corresponding to vertices of degree $1$ in the caterpillar, so that each such vertex is incident to exactly one new edge, and we also insert extra vertices along the base of the pyramid, while keeping the resulting graph simple and planar. The extra edges are inserted so that there are no vertices of degree $3$ in $G-u$ (i.e., degree $4$ in $G$), and so that there are vertices of degree $4$ in $G-u$ (i.e., degree $5$ in $G$). The extra vertices are inserted so that there are no $4$-cycles in $G-u$. Therefore, by Proposition \ref{prop:arad1} the constructed polyhedron $G$ satisfies $A(G)=\{1,2,4\}\cup A'_5$.
\begin{figure}[ht]
	\centering
	\includegraphics[width=6.cm]{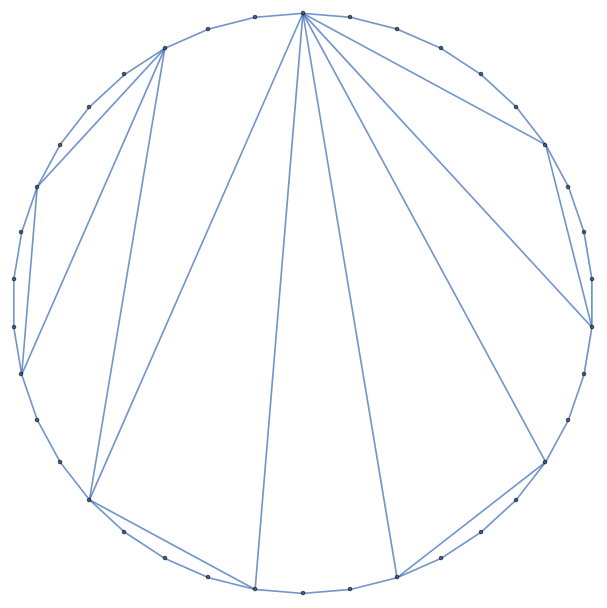}
	\caption{Proposition \ref{prop:124}: the graph $G-u$ for $A'_5=\{5,8\}$. The central vertices of the caterpillar have degrees $3,3,6$. The $3$ is repeated since $A'_5$ contains an odd number of odd elements.}
	\label{fig:124}
\end{figure}

As in Proposition \ref{prop:123}, we find infinite families of solutions by repeating and/or permuting the values of the central vertex degrees in the caterpillars.
\end{proof}

A possible code for Proposition \ref{prop:124}, in the Mathematica language, may be found in Appendix \ref{app:a}.

\begin{prop}
	\label{prop:012}
For every (finite) non-empty set $A'_3$ of integers $\geq 3$, there are infinitely many polyhedra $G$ satisfying
\[A(G)=\{0,1,2\}\cup A'_3.\]
\end{prop}
\begin{proof}
We consider the polyhedral graphs $G_{i,j}$ in Figure \ref{fig:012}, where $i\geq 1$ and
\[j=|N(u_{i,j},v_{i,j})|\geq 3.\]
\begin{figure}[ht]
	\centering
	\includegraphics[width=6.cm]{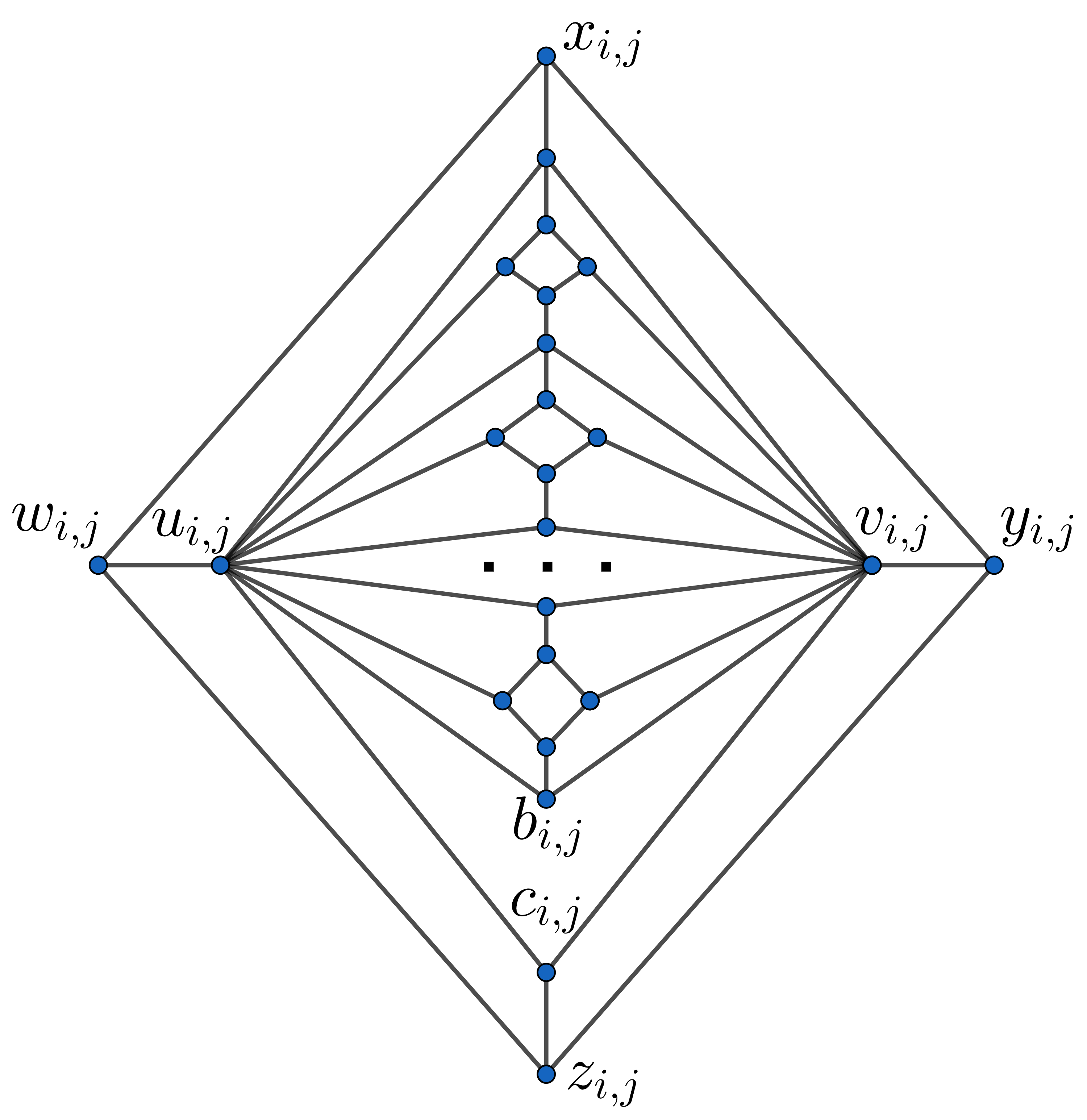}
	\caption{$G_{i,j}$.}
	\label{fig:012}
\end{figure}

Writing
\[A'_3=\{a_1,a_2,\dots,a_m\},\]
we take the graphs
\begin{equation}
\label{eq:Gi}
G_{1,a_1},G_{2,a_2},\dots,G_{m,a_m}.
\end{equation}
For $1\leq i\leq m-1$, we identify (i.e, `glue') the region
\[w_{i,a_i},x_{i,a_i},y_{i,a_i},z_{i,a_i},\]
of $G_{i,a_i}$ with the region 
\[u_{i+1,a_{i+1}},b_{i+1,a_{i+1}},v_{i+1,a_{i+1}},c_{i+1,a_{i+1}},\]
of $G_{i+1,a_{i+1}}$.

The resulting graph $G$ is a polyhedron, and
\[|N(u_{i,a_i},v_{i,a_i})|=a_i, \qquad 1\leq i\leq m,\]
while by construction all remaining pairs of vertices in $G$ have $0$, $1$, or $2$ common neighbours, so that
\[A(G)=\{0,1,2,a_1,a_2,\dots,a_m\},\]
as required.

To obtain infinitely many polyhedra of such type we modify the above construction as follows. Instead of \eqref{eq:Gi} we consider
\[G_{1,\alpha_1},G_{2,\alpha_2},\dots,G_{m',\alpha_{m'}}\]
where $m'\geq m$ is chosen arbitrarily, $\alpha_1,\alpha_2,\dots,\alpha_{m'}\in A'_3$, and each element of $A'_3$ appears at least once in $\{\alpha_1,\alpha_2,\dots,\alpha_{m'}\}$.
\end{proof}

\begin{proof}[Collecting the results for Theorem \ref{thm:2}]
Let us recollect the intermediate results that allowed us to prove Theorem \ref{thm:2}. They are: Corollary \ref{cor:01} (classification for the type $\{0,1\}$), Corollary \ref{cor:1234} (classification for $\{1,2\}$, $\{1,2,3\}$, and $\{1,2,4\}$), Corollary \ref{cor:012} (classification for the type $\{0,1,2\}$ when there are at least $25$ vertices), Proposition \ref{prop:034} (classification for $\{1,2,\ell\}$, even $\ell\geq 6$), Proposition \ref{prop:class} (proof that the list of types in Table \ref{t:2} is complete), Propositions \ref{prop:123}, \ref{prop:124} and \ref{prop:012} (respectively constructing infinitely many solutions for the types $\{1,2,3\}\cup A'_4$, $\{1,2,4\}\cup A'_5$, and $\{0,1,2\}\cup A'_3$).
\end{proof}

\paragraph{Future directions.} Given a class of (finite or infinite) graphs $\cg$, one may ask for which sets of non-negative integers $A$ there are infinitely many/finitely many/no elements in $\cg$ of type $A$. A more ambitious question is to classify the elements of $\cg$ according to their types. The quantity `collection of numbers of common neighbours for pairs (or more general subsets) of vertices' has vast potential for application in the sciences.

\appendix
\section{Code for Propositions \ref{prop:123} and \ref{prop:124}}
\label{app:a}

Figure \ref{fig:123code} presents possible code for Proposition \ref{prop:123}. The function called `dez' computes $A(G)$ for a graph $G$. The function `tr' initialises the vertices of the caterpillar for the inputted $A'_4$, `n' yields the base length of the pyramid, `ed' adds edges for one central vertex of the caterpillar, `de' repeats `ed' once for every central vertex of the caterpillar, `fi' deletes any duplicated edges, `fu' completes the construction of $G$, and `te' tests the output.
\begin{figure}[ht]
	\centering
	\includegraphics[width=\textwidth]{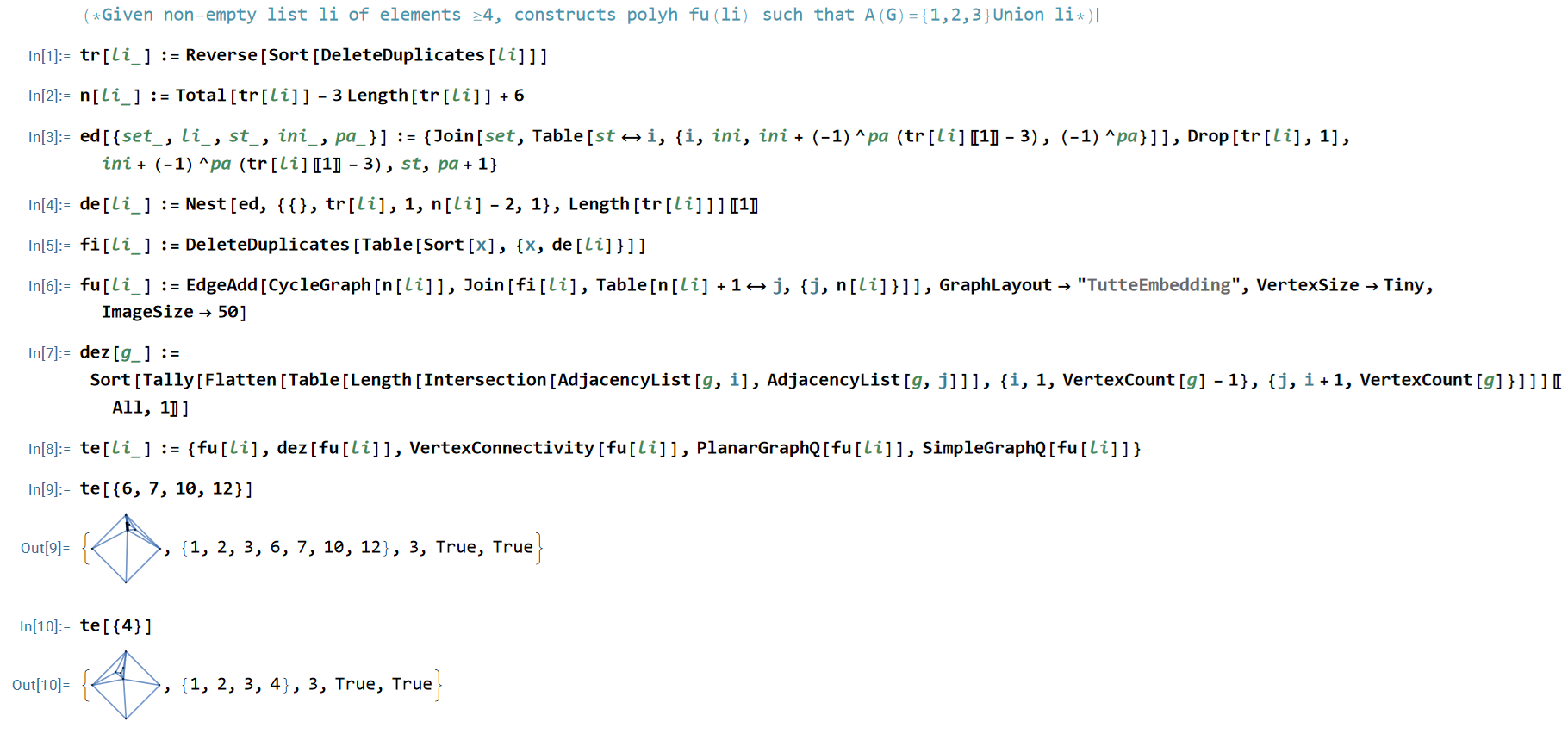}
	\caption{Code for Proposition \ref{prop:123}, in the Mathematica language.}
	\label{fig:123code}
\end{figure}

Figure \ref{fig:124code} presents possible code for Proposition \ref{prop:124}.
\begin{figure}[ht]
	\centering
	\includegraphics[width=\textwidth]{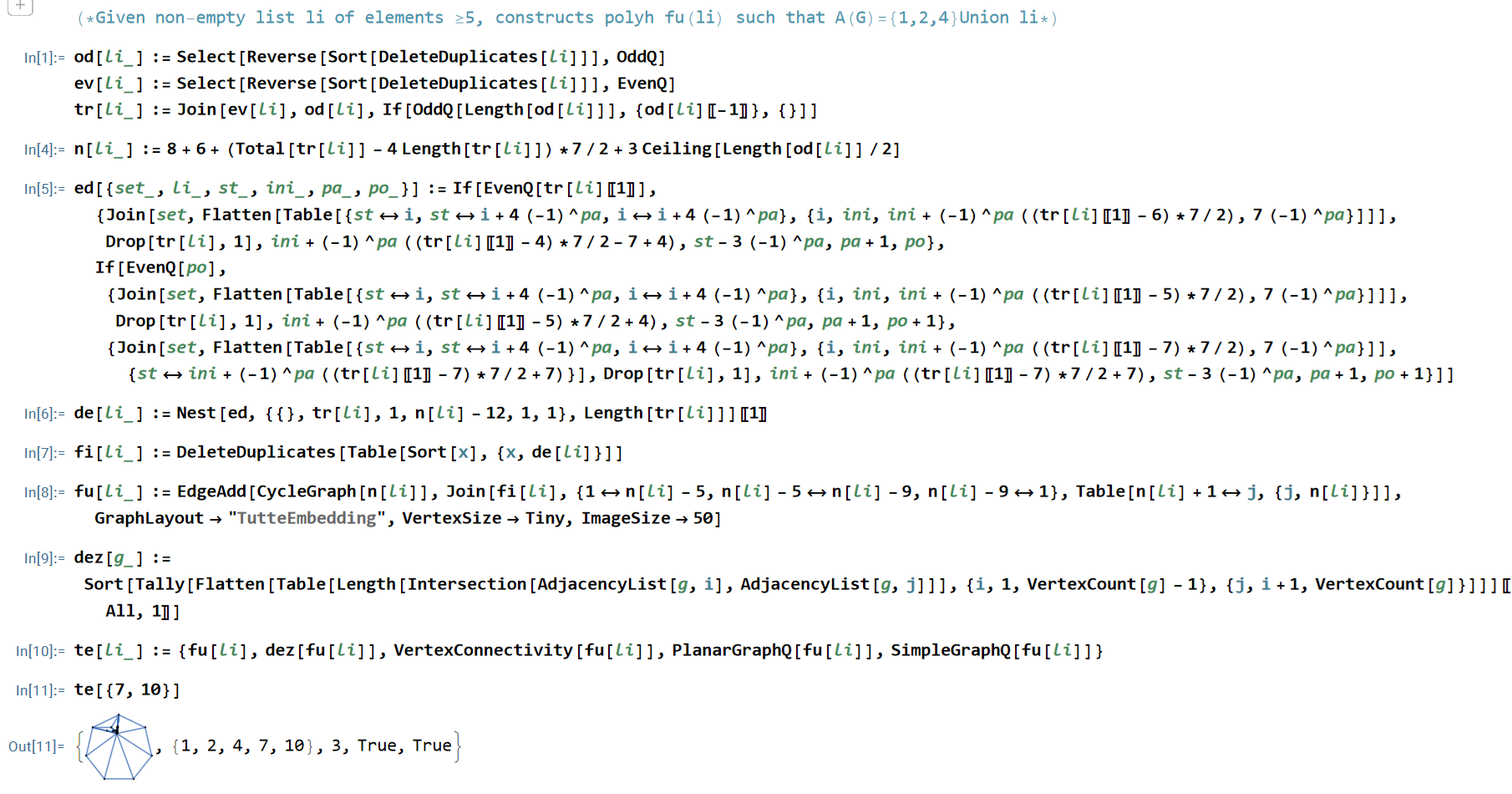}
	\caption{Code for Proposition \ref{prop:124}, in the Mathematica language.}
	\label{fig:124code}
\end{figure}

\paragraph{Acknowledgements.}
Riccardo W. Maffucci was partially supported by Programme for Young Researchers `Rita Levi Montalcini' PGR21DPCWZ \textit{Discrete and Probabilistic Methods in Mathematics with Applications}, awarded to Riccardo W. Maffucci.

\clearpage
\addcontentsline{toc}{section}{References}
\bibliographystyle{abbrv}
\bibliography{allgra}
\end{document}